\documentclass[a4paper,11pt,english]{smfart}
\usepackage{amsfonts}
\usepackage{amsmath,amsthm}
\usepackage{latexsym}
\usepackage{array}
\usepackage{amssymb}
\usepackage{graphicx}
\usepackage{enumerate}
\usepackage{verbatim}
\usepackage{float}
\usepackage[english]{babel}

\usepackage{color}
\definecolor{marin}{rgb}   {0.,   0.3,   0.7}
\definecolor{rouge}{rgb}   {0.8,   0.,   0.}
\definecolor{sepia}{rgb}   {0.8,   0.5,   0.}
\usepackage[colorlinks,citecolor=marin,linkcolor=rouge,
            bookmarksopen,
            bookmarksnumbered
           ]{hyperref}

\newcommand{\blue}{\color{black}}

\newcommand{\black}{\color{black}}

\theoremstyle{plain}
\newtheorem{theorem}{Theorem}[section]
\newtheorem{lemma}[theorem]{Lemma}

\newtheorem{corollary}[theorem]{Corollary}
 \theoremstyle{remark}
\newtheorem{remark}[theorem]{Remark}
\newtheorem{example}[theorem]{Example}

\newcommand {\aplt} {\ {\raise-.5ex\hbox{$\buildrel<\over\sim$}}\ }
\newcommand {\gplt} {\ {\raise-.5ex\hbox{$\buildrel>\over\sim$}}\ }

\newcommand{\p}{}
\newcommand{\m}{-}

\newcommand{\pp}{+}

\makeatletter
\def\@makefnmark{\hbox{$\m@th^{\@thefnmark}$}}
\makeatother

\author{Alexander Ostermann}
\address{Department of Mathematics, University of Innsbruck,
Technikerstr.~13, 6020 Innsbruck, Austria}
\email{alexander.ostermann@uibk.ac.at}

\author{Katharina Schratz}
\address{Fakult\"{a}t f\"{u}r Mathematik, Karlsruhe Institute of Technology,
Englerstr.~2, 76131 Karlsruhe, Germany}
\email{katharina.schratz@kit.edu}

\begin{abstract}
We introduce low regularity exponential-type integrators for nonlinear Schr\"odinger equations for which first-order convergence only requires the boundedness of one additional derivative of the solution. More precisely, we will prove first-order convergence in $H^r$ for solutions in $H^{r+1}$ ($r>d/2$) of the derived schemes. This allows us lower regularity assumptions on the data than for instance required for classical splitting or exponential integration schemes. For one dimensional quadratic Schr\"odinger equations we can even prove first-order convergence without any loss of regularity. Numerical experiments underline the favorable error behavior of the newly introduced exponential-type integrators for low regularity solutions compared to  classical splitting and exponential integration schemes.
\end{abstract}

\begin{document}

\title[Low regularity exponential-type integrators for Schr\"odinger equations]{ Low regularity exponential-type integrators for semilinear Schr\"odinger equations}

\maketitle
Semilinear Schr\"odinger equations, in particular those of type
\begin{equation}\label{nlsO}
i\partial_t u = \m  \Delta u + \mu \vert u \vert^{2p}u, \quad p \in \mathbb{N}
\end{equation}
with $\mu \in \mathbb{R}$ are nowadays extensively studied numerically. In this context, splitting methods (where the right-hand side is split into the kinetic and nonlinear part, respectively) as well as exponential integrators (including Lawson-type Runge--Kutta methods) contribute particularly attractive classes of integration schemes. For an extensive overview on splitting and exponential integration methods we refer to \cite{HLW,HochOst10,HLRS10,McLacQ02}, and for their rigorous convergence analysis in the context of semilinear Schr\"odinger equations we refer to \cite{BeDe02,CanG15,CCO08,CoGa12,Duj09,Faou12,Gau11,Lubich08,Ta12} and the references therein. However, within the construction of all these numerical methods the \emph{stiff part} (i.e., the terms involving the differential operator $\Delta$) is approximated in one way or another which generally requires the boundedness of two additional spatial derivatives of the exact solution. In particular, convergence of a certain order only holds under sufficient additional regularity assumptions on the solution.
 In the following we will illustrate the local error behavior of classical splitting and exponential integration methods and the thereby introduced smoothness requirements.\\

 \emph{Splitting schemes.}  The  Strang splitting for the  cubic  Schr\"odinger equation (i.e., $p=1$)
\begin{equation*}\label{split:cub}
\begin{aligned}
& u^{n+1/2}_{-} = \mathrm{e}^{ \p i \frac{\tau}{2} \Delta} u^n,\\
& u^{n+1/2}_{+} = \mathrm{e}^{-i \tau \mu\big\vert u^{n+1/2}_{-}\big\vert^2}  u^{n+1/2}_{-},
\\
& u^{n+1} = \mathrm{e}^{ \p i \frac{\tau}{2} \Delta}u^{n+1/2}_{+},
 \end{aligned}
\end{equation*}
where the right-hand side is split into the kinetic $T(u) = \p i \Delta u$ and  nonlinear part  $ V(u) = - i \mu\vert u\vert^2 u$ is rigorously analyzed in \cite{Lubich08}. In particular, its second-order convergence in $H^r$ for solutions in $H^{r+4}$ and its first-order convergence in $H^r$ for solutions in  $H^{r+2}$  holds for all $r \geq 0$, see \cite{Lubich08} and \cite{ESS16}, respectively. The result follows from the fact that the local error of the splitting scheme can be expressed through the double Lie commutator $[T,[T, V]](u)$ and the Lie commutator $[T,V](u)$ for second and first-order methods, respectively. The latter reads
$$
\frac{1}{2\mu} [T,V](u) =  \left(\nabla u \cdot \nabla \overline{u}\right) u + \left(\nabla u \overline{u} \right)\cdot \nabla u + \left(u \nabla \overline{u}\right)\cdot \nabla u + \left(u \overline{\Delta u} \right)u,
$$
see \cite[Section 4.2]{Lubich08}. Due to the appearance of $ \overline{\Delta u}$ in the local error, the boundedness of at least two additional derivatives of the exact solution is required. Based on  \cite{Lubich08}, fractional error estimates for the Strang splitting were established in \cite{ESS16} which require the boundedness of $2+2\gamma$ additional derivatives for convergence of order $1+\gamma$, with $0 < \gamma <1$. Furthermore, in \cite{Ignat11} first-order convergence of a filtered Lie splitting method for Schr\"odinger equations of type~\eqref{nlsO} with $1 \leq 2p < 4/d$ was shown  in $L^2(\mathbb{R}^d)$ for solutions in $H^2(\mathbb{R}^d)$.\\

\emph{Classical exponential integrators.} The classical first-order exponential integrator in the cubic case reads
\begin{equation}\label{exp:cub}
u^{n+1} = \mathrm{e}^{\p i \tau \Delta} u^n - i \mu \tau \varphi_1(\p i \tau \Delta) \Big( \vert u^n\vert^2 u^n\Big)\quad \text{with}\quad \varphi_1(z)  = \frac{\mathrm{e}^{z} - 1}{z}.
\end{equation}
Its construction is based on Duhamel's formula
\begin{equation*}
u(t_n+\tau) = \mathrm{e}^{ \p i \tau \Delta} u(t_n) - i \mu \int_0^\tau \mathrm{e}^{\p i (\tau-s)\Delta} \left \vert u(t_n+s)\right\vert^{2p}u(t_n+s) \mathrm{d}s
\end{equation*}
by applying the approximation
\begin{equation}
\label{uuu}
u(t_n+s) \approx u(t_n)
\end{equation}
in the integral terms and solving the remaining integral over the free Schr\"odinger group $\mathrm{e}^{\p i (\tau-s)\Delta}$ exactly. Due to the approximation \eqref{uuu} the above scheme introduces a local error of the form
\begin{equation*}
\begin{aligned}
\partial_t\bigl( i\vert u\vert^2 u \bigr) = i \overline{\partial_t u} u^2 + 2 i \vert u\vert^2 \partial_t u = \mu \vert u \vert^4 u \m 2  \vert u \vert^2\Delta u
\pp  u^2 \overline{\Delta u},
\end{aligned}
\end{equation*}
see \cite{HochOst10}. Hence, first-order convergence also requires the boundedness of at least two derivatives of the exact solution due to the appearance of $\Delta u$ and its complex conjugate counterpart.\\

The main novelty in this work lies in the development and analysis of a \emph{first-order exponential-type integrator} for Schr\"odinger equations of type \eqref{nlsO}
\begin{equation}\label{scheme}
u^{n+1} =  \mathrm{e}^{\p i \tau \Delta}\left[ u^n - i \mu \tau \big( u^n\big)^{p+1}\Big(\varphi_1(\m 2i \tau \Delta) \big( \overline{u^n}\big)^{p}\Big)\right] \quad \text{with} \quad \varphi_1(z) = \frac{\mathrm{e}^{z} - 1}{z}
\end{equation}
which \emph{only requires the boundedness of one additional derivative of the exact solution}. The construction is based on Duhamel's formula looking at the \emph{twisted} variable $v(t) = \mathrm{e}^{\m it \Delta } u(t)$ and treating the dominant term triggered by $\overline{u}^{2p}$ in an exact way. This idea of twisting the variable is widely used in the analysis of partial differential equations in low regularity spaces (see for instance \cite{Bour93}) and also well known in the context of numerical analysis, see \cite{Law67} for the introduction of Lawson-type Runge--Kutta methods. However, instead of approximating the appearing integrals with a Runge--Kutta method (see for instance \cite{KT05}) we \emph{integrate the dominant second-order stiff parts exactly}.

A similar approach has been successfully applied to the one dimensional KdV equation, see~\cite{HoS16}. However, due to the Burgers-type nonlinearity,  additional smoothness assumptions on the exact solution are necessary. More precisely, first-order convergence in $H^1$ is only guaranteed for solutions in $H^{3}$ in the KdV setting. The introduced exponential integrators for nonlinear Schr\"odinger equations are in contrast \emph{first-order convergent in $H^r$ for solutions in $H^{r+1}$} (with $r>d/2$). Furthermore, we show that our approach yields methods for quadratic nonlinearities of type $u^2$ and $\vert u \vert^2$ in one dimension which are convergent in $H^r$ for solutions in $H^r$, i.e., no additional smoothness of the solution is required. Such generous convergence results for nonlinear Schr\"odinger equations are up to our knowledge not yet know in the numerical analysis literature.

For practical implementation issues we impose periodic boundary conditions and refer to~\cite{Bour93} for local wellposedness (LWP) results of nonlinear Schr\"odinger equations in $H^r(\mathbb{T}^d)$. Note that in the case of a cubic nonlinearity we obtain LWP for $r > 1/2$ if $d \leq 3$. In the case of a quadratic nonlinearity of type $u^2$ (or $\overline{u}^2$) LWP can be pushed down to $r > - 1/2$ on the one dimensional torus, see~\cite{KPV96}. In the following let $r>d/2$. We denote by $\Vert \cdot \Vert_r$ the standard $H^r= H^r(\mathbb{T}^d)$ Sobolev norm.  In particular, we  exploit the well-known bilinear estimate
\begin{equation}\label{bil}
\Vert f g \Vert_r \leq c_{r,d} \Vert f \Vert_r \Vert g \Vert_r
\end{equation}
which holds for some constant $c_{r,d}>0$. Furthermore, we denote by $c$ a generic constant which may depend on $r$, $d$, $p$ and $\gamma$.
\section{Low regularity exponential integrators}
We consider nonlinear Schr\"odinger equations (NLS) of type
\begin{equation}\label{nls}
\begin{aligned}
& i \partial_t u(t,x) =\m \Delta  u(t,x) + \mu \vert u(t,x) \vert^{2p}u(t,x),\quad (t,x) \in \mathbb{R} \times \mathbb{T}^d, \quad p \in \mathbb{N}
\end{aligned}
\end{equation}
with $\mu \in \mathbb{R}$ and where we have set
$
\Delta = \sum_{k=1}^d\partial_{x_k}^2.
$
Furthermore, we employ the so-called $\varphi_1$ function defined as
$$
\varphi_1(z) = \frac{\mathrm{e}^{z} - 1}{z}.
$$
In the following we derive exponential-type integrators for Schr\"odinger equations of type \eqref{nls} based on iterating Duhamel's formula in the twisted variable $v(t) = \mathrm{e}^{\m i t\Delta } u(t)$. Note that the twisted variable  satisfies
\begin{equation}\label{nichtegal}
i \partial_t v(t) = \mu \mathrm{e}^{\m i t \Delta} \Big[ \vert  \mathrm{e}^{\p i t \Delta} v(t)\vert^{2p}  \mathrm{e}^{\p i t \Delta} v(t)\Big]
\end{equation}
with mild solution given by
\begin{equation}\label{vsol}
\begin{aligned}
v(t_n+\tau) = v(t_n) - i \mu \int_0^\tau \mathrm{e}^{\m i (t_n+s) \Delta} \Big[ \vert  \mathrm{e}^{\p i (t_n+s) \Delta} v(t_n+s)\vert^{2p}  \mathrm{e}^{\p i (t_n+s) \Delta} v(t_n+s)\Big]\mathrm{d}s.
\end{aligned}
\end{equation}
In order to derive our numerical scheme we proceed as follows. As $\mathrm{e}^{i t \Delta}$ is a linear isometry on $H^r$ for all $t \in \mathbb{R}$ and the bilinear estimate \eqref{bil} holds we obtain that
\begin{equation}\label{locapp}
\Vert v(t_n+s) - v(t_n)\Vert_r \leq |\mu| \int_0^s\Vert v(t_n+\xi)\Vert_r^{2p+1} \mathrm{d} \xi \leq s \vert \mu \vert  \sup_{0 \leq \xi \leq s} \Vert v(t_n+\xi)\Vert_r^{2p+1}, \quad r > d/2.
\end{equation}
In  this sense we have for $|s| \leq \tau$
\begin{equation}\label{locapp2}
v(t_n+s) \approx v(t_n)
\end{equation}
for a small time step $\tau$.
\begin{remark} Note that the twisted variable allows the formal expansion
\[
v(t_n+s) - v(t_n) = \mathcal{O}\left(s |v|^{2p}v\right)
\]
such that the approximation \eqref{locapp2} holds without additional regularity assumptions (see \eqref{locapp} for the rigorous estimate). The classical exponential integrator \eqref{exp:cub},  in contrast, is based on the approximation \eqref{uuu} for the original solution $u$. This approximation error is small only under additional smoothness assumptions, since
\[
u(t_n+s)  - u(t_n) = \left(\mathrm{e}^{i s \Delta}-1\right) u(t_n) + \mathcal{O}\left(s |u|^{2p}u\right).
\]
\end{remark}
Plugging \eqref{locapp2} into \eqref{vsol} yields  the approximation
\begin{equation}\label{vsolap1}
\begin{aligned}
v(t_n+\tau) &\approx v(t_n) - i \mu \int_0^\tau \mathrm{e}^{\m i (t_n+s) \Delta} \Big[ \vert  \mathrm{e}^{\p i (t_n+s) \Delta} v(t_n)\vert^{2p}  \mathrm{e}^{\p i (t_n+s) \Delta} v(t_n)\Big]\mathrm{d}s
\end{aligned}
\end{equation}
which is the basis of our numerical scheme.

Hence, we are left with deriving a numerical approximation to  the integral
$$
I_p^\tau(w,t_n) :=  \int_0^\tau \mathrm{e}^{\m i (t_n+s) \Delta} \Big[ \vert  \mathrm{e}^{\p i (t_n+s) \Delta} w\vert^{2p}  \mathrm{e}^{\p i (t_n+s) \Delta} w\Big]\mathrm{d}s.
$$
To illustrate the idea we will first consider the cubic case $p = 1$ in  Section~\ref{sec:cub} below. In Section~\ref{sec:p} we will deal with general nonlinearities $p \geq 1$. The special case of quadratic nonlinearities of type $u^2$ and $\vert u \vert^2$, respectively, will be treated in Section~\ref{sec:qnls}.

\subsection{Cubic nonlinearities \textbf{\textit{p}}${}={}$1}\label{sec:cub}
 For notational simplicity we first illustrate the idea in one dimension. In Section \ref{sec:cubd} we will give the generalization to arbitrary dimensions $d \geq 1$.
 \subsubsection{Cubic nonlinearities $p=1$ in dimension $d=1$} Let $f \in L^2(\mathbb{T})$. Then we will denote its Fourier expansion by $
f(x) = \sum_{k \in \mathbb{Z}} \hat{f}_k \mathrm{e}^{i k x}$. Furthermore, we  define a regularization of $\partial_x^{-1}$ through its action in Fourier space by
$$
(\partial_x^{-1})_k : =
\left\{
\begin{array}{ll}
(ik)^{-1} &\mbox{if} \quad k \neq 0\\
0 & \mbox{if}\quad k = 0
\end{array}
\right.,  \text{ i.e.,} \quad\partial_x^{-1} f(x) =\sum_{k\in \mathbb{Z}\setminus \{ 0\} } (ik)^{-1} \hat{f}_k \mathrm{e}^{ikx},
$$
and, by continuity,
\begin{equation}\label{conti}
\left(\frac{\mathrm{e}^{i t \partial_x^2} - 1}{i t \partial_x^2} \right)_{k\neq 0}  =
\frac{\mathrm{e}^{-i t k^2} - 1}{- i t k^2}, \qquad \left(\frac{\mathrm{e}^{i t \partial_x^2} - 1}{i t \partial_x^2} \right)_{k=0} = 1.
\end{equation}
In the case of a cubic nonlinearity (i.e.,  $p = 1$ in \eqref{nls}) in one spatial dimension the integral in~\eqref{vsolap1} can be expressed in terms of the Fourier expansion as follows
\begin{equation}\label{cubcal1}
\begin{aligned}
  I_1^\tau(w,t_n)  & =
 \int_0^\tau \mathrm{e}^{\m i (t_n+s) \partial_x^2} \Big[ \big( \mathrm{e}^{\m i (t_n+s) \partial_x^2} \overline{w}\big)  \big(\mathrm{e}^{\p i (t_n+s) \partial_x^2} w\big)^2\Big]\mathrm{d}s\\
 & =   \sum_{\ell \in \mathbb{Z}}  \sum_{\substack{k_1,k_2,k_3 \in \mathbb{Z}\\\ell = - k_1+k_2+k_3}} \int_0^\tau
 \mathrm{e}^{ i (t_n+s) \big[ \p \ell^2 \pp k_1^2 \m k_2^2 \m k_3^2\big] }\mathrm{d}s \,\overline{\hat{w}_{k_1}} \hat{w}_{k_2} \hat{w}_{k_3} \mathrm{e}^{i \ell x}\\
 & =    \sum_{\ell \in \mathbb{Z}}   \sum_{\substack{k_1,k_2,k_3 \in \mathbb{Z}\\\ell = - k_1+k_2+k_3}}\mathrm{e}^{ i t_n \big[ \p \ell^2 \pp k_1^2 \m k_2^2 \m k_3^2\big]}\overline{\hat{w}_{k_1}} \hat{w}_{k_2} \hat{w}_{k_3} \mathrm{e}^{i \ell x}  \int_0^\tau  \mathrm{e}^{ i s \big[ \p 2 k_1^2 \m 2k_1(k_2+k_3) \pp 2k_2 k_3\big] }\mathrm{d}s.
\end{aligned}
\end{equation}
Naturally, the next step would be the exact integration of the appearing integral
\begin{equation}\label{theintegral}
\int_0^\tau  \mathrm{e}^{ i s \big[ \p 2 k_1^2 \m 2k_1(k_2+k_3) \pp 2k_2 k_3\big]}\mathrm{d}s.
\end{equation}
However, the obtained convolution can not be converted into the physical space straightforwardly and therefore only yields a \emph{practical} scheme in dimension one as the iteration needs to be carried out fully in Fourier space. In order to obtain an \emph{efficient and practical} low regularity implementation we therefore only
$$
\text{treat the \emph{dominant} quadratic  term $\p 2 k_1^2$ exactly.}
$$
\blue For this purpose, we write the integrand in \eqref{theintegral} as
\begin{equation}\label{cubap2}
\begin{aligned}
 \mathrm{e}^{ i s \big[ \p 2 k_1^2 \m 2k_1(k_2+k_3) \pp 2k_2 k_3\big]} & =
 \mathrm{e}^{ \p 2 i s k_1^2}\left(1+ \vert \beta\vert^\gamma \,\frac{\mathrm{e}^{i\beta}-1}{\vert\beta\vert^\gamma} \right)
\end{aligned}
\end{equation}
with $\beta = s (\m 2k_1(k_2+k_3) \pp 2k_2 k_3\big)$ and some $0\le\gamma\le 1$\black. Inserting \eqref{cubap2} into \eqref{cubcal1} and integrating \blue the first term \black
$$
\int_0^\tau \mathrm{e}^{ 2 i s k_1^2}\mathrm{d} s = \tau \varphi_1(\p 2 i \tau k_1^2)
$$
yields that
\begin{equation}\label{cubcal2}
\begin{aligned}
 I_1^\tau(w,t_n)
& =\tau   \sum_{\ell \in \mathbb{Z}}   \sum_{\substack{k_1,k_2,k_3 \in \mathbb{Z}\\\ell = - k_1+k_2+k_3}}\mathrm{e}^{i t_n \big[ \p \ell^2 \pp k_1^2 \m k_2^2 \m k_3^2\big]} \varphi_1(\p 2 i \tau k_1^2)\overline{\hat{w}_{k_1}} \hat{w}_{k_2} \hat{w}_{k_3} \mathrm{e}^{i \ell x}+ \mathcal{R}^{\tau}_1(w,t_n)\\
& = \tau \mathrm{e}^{\m i t_n \partial_x^2} \Big[ \left( \mathrm{e}^{\p i t_n \partial_x^2} w\right)^2
\left( \varphi_1(\m 2 i \tau \partial_x^2)\mathrm{e}^{\m i t_n \partial_x^2} \overline{w}\right)
\Big]  +\mathcal{R}^{\tau}_1(w,t_n),
\end{aligned}
\end{equation}
\blue where \black
\begin{equation*}
\begin{aligned}
\Vert \mathcal{R}^{\tau}_1(w,t_n)\Vert_r^2
&= \sum_{\ell \in \mathbb{Z}} \left(1+\vert \ell \vert\right)^{2r} \Bigg \vert \sum_{\substack{k_1,k_2,k_3 \in \mathbb{Z}\\\ell = - k_1+k_2+k_3}}  \mathrm{e}^{ i t_n \big[ \p \ell^2 \pp k_1^2 \m k_2^2 \m k_3^2\big]}\big \vert 2 k_1(k_2+k_3) - 2k_2k_3\big\vert^\gamma\\
&\qquad \overline{\hat{w}_{k_1}} \hat{w}_{k_2} \hat{w}_{k_3}\mathrm{e}^{i \ell x}\int_0^\tau  \mathrm{e}^{ 2 i s k_1^2 }   \Bigg(\frac{ \mathrm{e}^{i s \big[ -2 k_1(k_2+k_3) + 2k_2k_3\big]}-1}{s^\gamma \vert 2 k_1(k_2+k_3) - 2k_2k_3\big\vert^\gamma}\Bigg)
 s^\gamma  \mathrm{d}s  \Bigg \vert^2.
\end{aligned}
\end{equation*}
\blue Since $|\beta|^{-\gamma}(\mathrm{e}^{i\beta}-1)$ is uniformly bounded for $\beta\in \mathbb{R}$ and $0\leq \gamma \leq 1$, the integral in the above formula is of the order $\tau^{1+\gamma}$. This shows that \black
\begin{equation}
\begin{aligned}\label{bremOf}
\Vert \mathcal{R}^{\tau}_1(w,t_n)\Vert_r^2  \le c\tau^{2+2\gamma}\sum_{\ell \in \mathbb{Z}} \left(1+\vert \ell \vert\right)^{2r} \Big ( \hspace{-2mm}\sum_{\substack{k_1,k_2,k_3 \in \mathbb{Z}\\\ell = - k_1+k_2+k_3}} \!\!\blue\big( \vert k_1 k_2 \vert +\vert k_1 k_3 \vert +\vert k_2 k_3\vert \big)^\gamma \black\,
\vert \hat{w}_{k_1} \vert \vert \hat{w}_{k_2}\vert \vert \hat{w}_{k_3}\vert \Big )^2\\
\blue \le c\tau^{2+2\gamma}\sum_{\ell \in \mathbb{Z}} \left(1+\vert \ell \vert\right)^{2r} \Big ( \hspace{-2mm} \sum_{\substack{k_1,k_2,k_3 \in \mathbb{Z}\\\ell = - k_1+k_2+k_3}} \!\!  \big (1+\vert k_1 \vert\big)^\gamma \big (1+\vert k_2 \vert\big)^\gamma \big (1+\vert k_3 \vert\big)^\gamma\,
\vert \hat{w}_{k_1}\vert \vert \hat{w}_{k_2}\vert \vert \hat{w}_{k_3}\vert \Big )^2
\end{aligned}
\end{equation}
for some constant $c>0$ \blue and $0 \leq \gamma \leq 1$.

Next we define the auxiliary  function
$g(x) = \sum_{k\in \mathbb{Z}} \hat{g}_k\mathrm{e}^{ikx}$ through its Fourier coefficients
\begin{align}\label{Fu}
\hat{g}_k = (1+\vert k \vert)^\gamma \vert \hat w_k\vert
\end{align}
which allows us to express the bound on the remainder given in \eqref{bremOf} as follows
\begin{align*}
\Vert \mathcal{R}^{\tau}_1(w,t_n)\Vert_r \leq c \tau^{1+\gamma} \Vert g^3\Vert_r.
\end{align*}
With the aid of the bilinear estimate \eqref{bil} we can thus conclude that for $r>1/2$
\begin{align*}
    \Vert \mathcal{R}^{\tau}_1(w,t_n)\Vert_r &  \leq c \tau^{1+\gamma} \Vert g\Vert_r^3  = c \tau^{1+\gamma} \left(\sum_{k \in \mathbb{Z}} (1+\vert k\vert)^{2(r+\gamma)} \vert \hat w_k\vert^2\right)^{3/2},
\end{align*}
where we have used the definition of the Fourier coefficients in \eqref{Fu}. Hence, we can conclude the following bound on the remainder
\black
\begin{equation}\label{rem1cub}
\Vert \mathcal{R}^{\tau}_1(w,t_n)\Vert_r \leq c \tau^{1+\gamma} \Vert w\Vert_{r+\gamma}^{3}\quad \text{ for } \quad 0 \leq \gamma \leq 1, \quad r > 1/2
\end{equation}
for some constant $c>0$.

Inserting the  dominant term of   \eqref{cubcal2} into \eqref{vsolap1} and twisting the solution back again, i.e., setting
$$u^n := \mathrm{e}^{\p i t_n \partial_x^2} v^n$$
yields the following exponential-type integration scheme for the cubic NLS
  \begin{equation}\label{cubscheme}
  u^{n+1} = \mathrm{e}^{\p i \tau \partial_x^2} \left[ u^n - i \tau \mu (u^n)^2\Big( \varphi_1(\m 2 i \tau \partial_x^2)\overline{u^n}\Big)\right].
  \end{equation}
 Next we generalize the above scheme to arbitrary dimensions $d \geq 1$.
   \subsubsection{Cubic nonlinearities $p=1$ in dimension $d\geq 1$} \label{sec:cubd}
   In the following we use the notation
   $$\mathbf{j} := (j_1,\ldots,j_d)\in \mathbb{Z}^d,\quad \mathbf{j}\cdot \mathbf{x}:= j_1 x_1 + \ldots + j_d x_d.$$
   In the case of a cubic nonlinearity $p = 1$ in dimension $d$ the integral in~\eqref{vsolap1} can be expressed in terms of the Fourier expansion  as follows
\begin{equation*}\label{cubcal1d}
\begin{aligned}
 I_1^\tau(w,t_n) &=
 \int_0^\tau \mathrm{e}^{\m i (t_n+s) \Delta} \Big[ \big( \mathrm{e}^{\m i (t_n+s)  \Delta} \overline{w}\big)  \big(\mathrm{e}^{\p i (t_n+s)  \Delta} w\big)^2\Big]\mathrm{d}s
 \\&=  \sum_{\mathbf{\Omega} \in \mathbb{Z}^d}
  \sum_{\substack{\mathbf{j},\mathbf{k},\mathbf{l}\in \mathbb{Z}^d\\\mathbf{\Omega} = -\mathbf{j}+\mathbf{k}+\mathbf{l}}} \int_0^\tau \mathrm{e}^{i (t_n+s) \big[\p \mathbf{\Omega}^2 \pp \mathbf{j}^2 \m \mathbf{k}^2 \m \mathbf{l}^2\big]} \mathrm{d}s\,\overline{\hat{w}_{\mathbf{j}}}\hat{w}_{\mathbf{k}}\hat{w}_{\mathbf{l}}\mathrm{e}^{i \,\mathbf{\Omega}\cdot \mathbf{x}}.
\end{aligned}
\end{equation*}
Note that we have
$$
\p \mathbf{\Omega}^2 \pp \mathbf{j}^2 \m \mathbf{k}^2 \m \mathbf{l}^2 = \p 2 \mathbf{j}^2 \m 2\mathbf{j}\cdot (\mathbf{k}+\mathbf{l}) \pp 2 \mathbf{k} \cdot \mathbf{l}.
$$
In order to obtain an \emph{efficient and practical} low regularity implementation we again
$$
\text{treat the \emph{dominant} quadratic  term $\p 2\mathbf{j}^2$ exactly}.
$$
This yields that
\begin{equation}\label{cubcal2d}
\begin{aligned}
 I_1^\tau(w,t_n)&=
 \sum_{\mathbf{\Omega} \in \mathbb{Z}^d}  \sum_{\substack{\mathbf{j},\mathbf{k},\mathbf{l}\in \mathbb{Z}^d\\\mathbf{\Omega} = -\mathbf{j}+\mathbf{k}+\mathbf{l}}}  \mathrm{e}^{i t_n \big[\p \mathbf{\Omega}^2 \pp \mathbf{j}^2\m\mathbf{k}^2\m\mathbf{l}^2\big]}\overline{\hat{w}_{\mathbf{j}}}\hat{w}_{\mathbf{k}}\hat{w}_{\mathbf{l}}\mathrm{e}^{i \,\mathbf{\Omega}\cdot \mathbf{x}}\int_0^\tau \mathrm{e}^{\p 2i s \mathbf{j}^2}\mathrm{d}s+ \mathcal{R}_1^\tau(w,t_n)\\
&= \sum_{\mathbf{\Omega} \in \mathbb{Z}^d} \sum_{\substack{\mathbf{j},\mathbf{k},\mathbf{l}\in \mathbb{Z}^d\\\mathbf{\Omega} = -\mathbf{j}+\mathbf{k}+\mathbf{l}}}  \mathrm{e}^{i t_n \big[\p\mathbf{\Omega}^2 \pp \mathbf{j}^2\m\mathbf{k}^2\m\mathbf{l}^2\big]}\overline{\hat{w}_{\mathbf{j}}}\hat{w}_{\mathbf{k}}\hat{w}_{\mathbf{l}}\mathrm{e}^{i \,\mathbf{\Omega}\cdot \mathbf{x}}\varphi_1(\p 2i \tau \mathbf{j}^2)+ \mathcal{R}_1^\tau(w,t_n),
\end{aligned}
\end{equation}
where similarly to above   (see \eqref{rem1cub})   we have that
\begin{equation*}\label{rem1cubd}
\Vert \mathcal{R}_1^{\tau}(w,t_n)\Vert_r \leq c \tau^{1+\gamma} \Vert w\Vert_{r+\gamma}^{3}\quad \text{ for } \quad 0 \leq \gamma \leq 1, \quad r > d/2
\end{equation*}
for some constant $c>0$.

Inserting the approximation \eqref{cubcal2d} into \eqref{vsolap1} and twisting the solution back again, i.e., setting
$$u^n := \mathrm{e}^{\p i t_n \Delta} v^n$$
yields the generalization of  the exponential-type integrator \eqref{cubscheme} to dimensions $d\geq 1$
\begin{equation*}\label{schemecubd}
  u^{n+1} = \mathrm{e}^{\p i \tau \Delta} \left[ u^n - i \tau \mu (u^n)^2\Big( \varphi_1(\m 2 i \tau \Delta)\overline{u^n}\Big)\right].
\end{equation*}

Next we consider the general case of \eqref{nls}  with $p \in \mathbb{N}$.
\subsection{Nonlinearities with integer \textbf{\textit{p}}}\label{sec:p}
 In the case   of a general nonlinearity   $p \in \mathbb{N}$ in \eqref{nls} the integral in \eqref{vsolap1} can be expressed in terms of the Fourier expansion as follows
\begin{equation*}
\begin{aligned}\label{cal1}
 I_p^\tau&(w,t_n)=
 \int_0^\tau \mathrm{e}^{\m i (t_n+s) \Delta} \Big[ \vert  \mathrm{e}^{\p i (t_n+s) \Delta} w\vert^{2p}  \mathrm{e}^{\p i (t_n+s) \Delta} w\Big]\mathrm{d}s.
 \end{aligned}
\end{equation*}
Similar observations to above  (see also \cite{SchO16}) yield the following numerical scheme
\begin{equation}\label{schemeP}
v^{n+1} =  v^n  - i \mu \tau \mathrm{e}^{\m i t_n \Delta} \left[ \left( \mathrm{e}^{\p i t_n \Delta} v^n\right)^{p+1} \left(\varphi_1(\m 2i \tau \Delta) \left(  \mathrm{e}^{\m i t_n \Delta}  \overline{v^n}\right)^{p}\right)\right]
\end{equation}
which satisfies
\begin{equation}\label{cal2}
v^{n+1} = v^n - i \mu I_p^\tau(v^n,t_n) -i \mu  \mathcal{R}^\tau_p(v^n,t_n),
\end{equation}
with
\begin{equation}\label{rem1}
\Vert \mathcal{R}_p^{\tau}(v^n,t_n)\Vert_r \leq c \tau^{1+\gamma} \Vert v^n\Vert_{r+\gamma}^{2p+1}\quad \text{ for } \quad 0 \leq \gamma \leq 1, \quad r > d/2
\end{equation}
for some constant $c>0$.

In the original variable $u$ the exponential-type integration scheme for the nonlinear Schr\"odinger equation \eqref{nls} reads
\begin{equation}\label{schemePu}
u^{n+1} =  \mathrm{e}^{\p i \tau \Delta}\left[ u^n - i \mu \tau \big( u^n\big)^{p+1}\Big(\varphi_1(\m 2i \tau \Delta) \big( \overline{u^n}\big)^{p}\Big)\right].\end{equation}
In the following section we give an error analysis for the above scheme.

\section{Error analysis}\label{sec:err}
In this section we carry out the error analysis of the exponential-type integrator \eqref{schemeP}. In the following, we set
\begin{equation*}\label{flowP}
\Phi_t^\tau(w) := w - i \mu \tau \mathrm{e}^{\m i t \Delta} \left[ \left( \mathrm{e}^{\p i t \Delta} w\right)^{p+1} \Big(\varphi_1(\m 2 i \tau \Delta)\left(  \mathrm{e}^{\m i t \Delta}  \overline{w}\right)^{p}\Big)\right]
\end{equation*}
such that in particular $v^{k+1} = \Phi_{t_k}^\tau(v^k)$.

\begin{lemma}[Stability]\label{lem:stabP}
Let $r>d/2$ and $f,g\in H^r$. Then, for all $t \in \mathbb{R}$ we have
\begin{equation*}
\Vert \Phi_t^\tau(f) - \Phi_t^\tau(g)\Vert_r \leq (1+ \tau \vert \mu \vert L) \Vert f - g \Vert_r,
\end{equation*}
where $L$ depends on $\Vert f \Vert_r$ and $\Vert g \Vert_r$.
\end{lemma}
\begin{proof}
Note that for all $t \in \mathbb{R}$ and $w \in H^r$ it holds by \eqref{conti} that
$$
\big\Vert \varphi_1(2 i t \Delta) w \big\Vert_r  \leq c \Vert w \Vert_r.
$$
Thus, as $\mathrm{e}^{i t \Delta}$ is a linear isometry on $H^r$ and $H^r$ is an algebra for $r > d/2$ we obtain that
\begin{equation*}\label{eq:stabP}
\begin{aligned}
\Vert \Phi_t^\tau(f) - \Phi_t^\tau(g)\Vert_r &  \leq \Vert f - g \Vert_r + \tau \vert \mu\vert \Big \Vert
\left( \mathrm{e}^{\p i t \Delta} f\right)^{p+1} \left( \varphi_1(\m 2 i \tau \Delta)\left(  \mathrm{e}^{\m i t \Delta}  \overline{f}\right)^{p}\right) \\&\qquad\qquad\qquad\qquad\qquad- \left( \mathrm{e}^{\p i t \Delta} g\right)^{p+1} \left( \varphi_1(\m 2 i \tau \Delta) \left(  \mathrm{e}^{\m i t \Delta}  \overline{g}\right)^{p}\right)\Big\Vert_r\\
& \leq  \Vert f - g \Vert_r + \tau   c  \vert \mu\vert  \Vert f-g\Vert_r \sum_{l=0}^{2p} \Vert f \Vert_r^{2p-l} \Vert g \Vert_r^{l}.
\end{aligned}
\end{equation*}
This proves the assertion.
\end{proof}
\begin{lemma}[Local error]\label{lem:locP}
Let $r>d/2$ and $0 \leq \gamma \leq 1$.  Assume that $v(t_k+t) = \phi^t(v(t_k)) \in H^{r+\gamma}$ for $0 \leq t \leq \tau$, where $\phi^t$ denotes the exact flow of \eqref{nichtegal}.   Then,
\begin{equation*}
\Vert \phi^\tau(v(t_k)) - \Phi_{t_k}^\tau(v(t_k))\Vert_r \leq c \tau^{1+\gamma},
\end{equation*}
where $c$ depends on $\sup_{0 \leq t \leq \tau} \Vert \phi^t(v(t_k))\Vert_{r+\gamma}$.
\end{lemma}

\begin{proof}
Note that by \eqref{cal2} we have that
\begin{equation*}
\begin{aligned}
\Phi_{t_k}^\tau(v(t_k)) & = v(t_k) - i \mu \int_0^\tau \mathrm{e}^{\m i (t_k+s) \Delta} \Big[ \vert  \mathrm{e}^{\p i (t_k+s) \Delta} v(t_k)\vert^{2p}  \mathrm{e}^{\p i (t_k+s) \Delta} v(t_k)\Big]\mathrm{d}s + i \mu  \mathcal{R}^\tau_p(v(t_k),t_k).
\end{aligned}
\end{equation*}
Furthermore,   \eqref{vsol} implies   that
$$
\phi^\tau(v(t_k))   = v(t_k) - i \mu \int_0^\tau \mathrm{e}^{\m i (t_k+s) \Delta} \Big[ \vert  \mathrm{e}^{\p i (t_k+s) \Delta} v(t_k+s)\vert^{2p}  \mathrm{e}^{\m i (t_k+s) \Delta} v(t_k+s)\Big]\mathrm{d}s.
$$
Thus as $\mathrm{e}^{i t \Delta}$ is a linear isometry on $H^r$ and $H^r$ is an algebra for $r > d/2$ we obtain that
\begin{equation*}\label{le}
\begin{aligned}
\Vert \phi^\tau(v(t_k)) - \Phi_{t_k}^\tau(v(t_k))\Vert_r &\leq \vert \mu\vert \sup_{0 \leq s \leq \tau} \Vert v(t_{k}+s)\Vert_r ^{2p} \int_0^\tau
\Vert v(t_{k}+s) - v(t_{k})\Vert_r \,
 \mathrm{d}s \\&\qquad\qquad+   \vert \mu \vert   \Vert \mathcal{R}_p^\tau(v(t_k),t_k)\Vert_r.
\end{aligned}
\end{equation*}
  Together with the estimate on the difference $v(t_k+s) - v(t_k)$ given  in \eqref{locapp} and the estimate on the remainder $ \mathcal{R}_p^\tau(v,t)$ given in \eqref{rem1} this yields for $r>d/2$ and $0 \leq \gamma \leq $  that
\begin{equation*}
\begin{aligned}
\Vert \phi^\tau(v(t_k)) - \Phi_{t_k}^\tau(v(t_k))\Vert_r \leq  \tau^2 c\vert \mu \vert^2  \sup_{0 \leq s \leq \tau} \Vert v(t_k+s)\Vert_r^{4p+1} +   \tau^{1+\gamma} c\vert \mu \vert  \Vert v(t_k)\Vert_{r+\gamma}^{2p+1}
\end{aligned}
\end{equation*}
which concludes the  stated local error bound.
\end{proof}
These two lemmata allow us to prove the following convergence bound.
\begin{theorem}\label{thm1P}
Let $r>d/2$ and $0   <   \gamma \leq 1$. Assume that the exact solution of \eqref{nichtegal} satisfies $v(t) \in H^{r+\gamma}$ for $0 \leq t \leq T$. Then, there exists a  constant   $\tau_0>0$ such that for all  step sizes   $\tau \leq \tau_0$ and $t_n \leq T$ we have that the global error of \eqref{schemeP} is bounded by
\begin{equation*}
\Vert v(t_{n}) - v^{n} \Vert_r \leq c \tau^\gamma,
\end{equation*}
where $c$ depends on $\sup_{0\leq t \leq T} \Vert v(t)\Vert_{r+\gamma}$.
\end{theorem}
\begin{proof}
The triangle inequality yields that
\begin{equation*}\label{inqP}
\begin{aligned}
\Vert v(t_{k+1}) - v^{k+1}\Vert_{r} & = \Vert \phi^\tau(v(t_k) )- \Phi^\tau_{t_k}(v^k)\Vert_{r} \\
& \leq \Vert  \phi^\tau(v(t_k)) - \Phi^\tau_{t_k}(v(t_k)) \Vert_{r} + \Vert \Phi^\tau_{t_k}(v(t_k)) - \Phi^\tau_{t_k}(v^k)  \Vert_{r}.
\end{aligned}
\end{equation*}
  Thus, we obtain with the aid of Lemmata \ref{lem:locP} and \ref{lem:stabP} that as long as $ v^{k} \in H^r$ we have that
\begin{equation}\label{inq2}
\begin{aligned}
\Vert v(t_{k+1}) - v^{k+1}\Vert_{r} & \leq c \tau^{1+\gamma} +\left(1+\tau \vert \mu\vert L\right) \Vert v(t_k)-v^k\Vert_{r},
\end{aligned}
\end{equation}
where $c$ depends on $\mathrm{sup}_{t_k \leq t \leq t_{k+1}} \Vert v(t)\Vert_{r+\gamma}$ and $L$ depends on $\Vert v(t_k)\Vert_{r}$ as well as on $ \Vert v^{k} \Vert_{r}$.   As long as $v(t_k) \in H^{r+\gamma}$ and $v^k \in H^r$ for $0 \leq k \leq n$ we therefore obtain by iterating the estimate~\eqref{inq2} that
\begin{equation*}
\begin{aligned}
\Vert v(t_{n+1}) - v^{n+1}\Vert_{r} &  \leq c \tau^{1+\gamma} +\left(1+\tau \vert \mu\vert L\right) \Vert v(t_n)-v^n\Vert_{r} \\
& \leq  c \tau^{1+\gamma} + \mathrm{e}^{\tau  \vert \mu\vert L} \left(
 c \tau^{1+\gamma}  + \left(1+\tau \vert \mu\vert L\right) \Vert v(t_{n-1})-v^{n-1}\Vert_{r}
\right)\\
& \leq c n \tau^{1+\gamma} \mathrm{e}^{n \tau \vert \mu \vert L} \leq c t_n \tau^\gamma \mathrm{e}^{t_n \vert \mu \vert L}.
\end{aligned}
\end{equation*}
 The assertion then follows by \blue induction\black , respectively, \blue a \black \emph{Lady Windermere's fan} argument (see, for example \cite{Faou12,HNW93,Lubich08}).
\end{proof}
The above theorem immediately yields a convergence result for the exponential-type integration scheme~\eqref{schemeP} approximating the solution of the nonlinear Schr\"odinger equation \eqref{nls}.
\begin{corollary}\label{coru1P}
Let $r>d/2$ and $0   <   \gamma \leq 1$. Assume that the exact solution of \eqref{nls} satisfies $u(t) \in H^{r+\gamma}$ for $0 \leq t \leq T$. Then, there exists a  constant   $\tau_0>0$ such that for all  step sizes   $\tau \leq \tau_0$ and $t_n \leq T$ we have that the global error of \eqref{schemePu} is bounded by
\begin{equation*}
\Vert u(t_{n}) - u^{n} \Vert_r \leq c \tau^\gamma,
\end{equation*}
where $c$ depends on $\sup_{0\leq t \leq T} \Vert u(t)\Vert_{r+\gamma}$.
\end{corollary}
\begin{proof}
As $\mathrm{e}^{i t \Delta}$ is a linear isometry on $H^r$ for all $t \in \mathbb{R}$
we have that
$$
\Vert u(t_{n}) - u^{n} \Vert_r = \Vert \mathrm{e}^{\p i t_{n} \Delta} \big(v(t_{n}) - v^{n} \big) \Vert_r = \Vert v(t_{n}) - v^{n} \Vert_r.
$$
Thus, the assertion follows from Theorem~\ref{thm1P}.
\end{proof}

\section{Quadratic Schr\"odinger equations in one space dimension}\label{sec:qnls}
In this section we focus on quadratic nonlinear Schr\"odinger equations of type (see, e.g., \cite{BeTao06})
\begin{align}\label{eq:qnls1}
i \partial_t u = \m \partial_x^2 u + \mu u^2 ,\quad u(0,x) = u_0(x),\quad  (t,x) \in \mathbb{R} \times \mathbb{T},\quad \mu \in \mathbb{R}
\end{align}
as well as of type (see, e.g, \cite{KPV96})
\begin{align}\label{eq:qnls2}
i \partial_t u = \m  \partial_x^2 u +\mu \vert u\vert^2 ,\quad u(0,x) = u_0(x),\quad  (t,x) \in \mathbb{R} \times \mathbb{T},\quad \mu \in \mathbb{R},
\end{align}
which have recently gained a lot of attention in literature, see for instance \cite{Bour93,BeTao06,CaWe90,GerMS09,KPV96,KPV01,Kish09,NHT01,Tao01,Ts87} and the references therein.

Here, the key-relations
\begin{equation}\label{eq:keyd}
\p (k_1+k_2)^2 \m k_1^2 \m k_2^2 = \p 2k_1 k_2
\end{equation}
and
\begin{equation}\label{eq:keyd2}
\p (k_1-k_2)^2 \m k_1^2 \pp k_2^2 =  \m2k_2 (k_1- k_2)
\end{equation}
allow us to derive \emph{first-order exponential-type integrators} for Schr\"odinger equations of type \eqref{eq:qnls1} and \eqref{eq:qnls2} \emph{without imposing any additional regularity assumptions on the solution}.

\subsection{Quadratic Schr\"odinger equations of type (\ref{eq:qnls1})}\label{secQNLS1}
Set $v(t) = \mathrm{e}^{\m i t\partial_x^2 } u(t)$, where $u$ is the solution of the Schr\"odinger equation \eqref{eq:qnls1}. Then the twisted variable $v(t)$ satisfies
\begin{equation}\label{eq:v1}
i \partial_t v = \mu  \mathrm{e}^{\m i t\partial_x^2 }  \left( \mathrm{e}^{\p i t\partial_x^2 } v(t)\right)^2
\end{equation}
which by Duhamel's formula yields that
\begin{equation}
\begin{aligned}\label{eq:vdnls}
v(t_n+\tau) = v(t_n) - i \mu \int_0^\tau \mathrm{e}^{\m i (t_n+s) \partial_x^2} \left ( \mathrm{e}^{\p i (t_n+s) \partial_x^2 } v(t_n+s)\right)^2\mathrm{d}s.
\end{aligned}
\end{equation}
To construct our numerical scheme we proceed as above. We start from the following approximation
\begin{equation}\label{approxv1}
\begin{aligned}
v(t_n+\tau)
& \approx  v(t_n) - i \mu \int_0^\tau \mathrm{e}^{\m i (t_n+s) \partial_x^2}  \left(\mathrm{e}^{\p i  (t_n+s)\partial_x^2 } v(t_n)\right)^2\mathrm{d}s =: \Phi_{t_n}^\tau(v(t_n))
\end{aligned}
\end{equation}
and rewrite the above integral in Fourier space as follows
\begin{equation*}
\begin{aligned}
 I^\tau(w,t_n) & =  \int_0^\tau \mathrm{e}^{\m i  (t_n+s)\partial_x^2 }  \left( \mathrm{e}^{\p i  (t_n+s)\partial_x^2 } w\right)^2\mathrm{d}s \\& = \int_0^\tau  \mathrm{e}^{\m i  (t_n+s)\partial_x^2 } \Big[ \Big( \mathrm{e}^{\p i  (t_n+s)\partial_x^2 } \sum_{k_1} \hat{w}_{k_1} \mathrm{e}^{i k_1 x}\Big)\Big( \mathrm{e}^{\p i  (t_n+s)\partial_x^2 }  \sum_{k_2} \hat{w}_{k_2}\mathrm{e}^{i k_2 x}\Big)\Big] \mathrm{d}s\\
 & = \int_0^\tau \sum_{k_1,k_2} \mathrm{e}^{ i (t_n+s)\big[\p (k_1+k_2)^2 \m k_1^2 \m k_2^2\big]}  \hat{w}_{k_1}\hat{w}_{k_2} \mathrm{e}^{i (k_1+k_2)x}\mathrm{d}s.
 \end{aligned}
 \end{equation*}
 The key-relation \eqref{eq:keyd} now shows that
\begingroup
\allowdisplaybreaks
\begin{align*}
 I^\tau(w,t_n)
 & =  \sum_{\substack{k_1,k_2 \\ k_1, k_2 \neq 0}} \mathrm{e}^{ i t_n\big[\p (k_1+k_2)^2 \m k_1^2 \m k_2^2\big]} \frac{ \mathrm{e}^{ i \tau\big[\p (k_1+k_2)^2 \m k_1^2 \m k_2^2\big]}-1}{\p 2 i k_1 k_2}  \hat{w}_{k_1}\hat{w}_{k_2}\mathrm{e}^{i (k_1+k_2)x}\\
 &\qquad + 2 \tau \hat{w}_0 \sum_{k_1\neq 0} \hat{w}_{k_1} \mathrm{e}^{i k_1 x} + \tau \hat{w}_0^2 \\
  &    = \p \frac{i}{2} \sum_{\substack{k_1,k_2 \\ k_1, k_2 \neq 0}} \mathrm{e}^{ i t_n\big[\p (k_1+k_2)^2 \m k_1^2 \m k_2^2\big]} \frac{ \mathrm{e}^{ i \tau\big[\p (k_1+k_2)^2 \m k_1^2 \m k_2^2\big]}-1}{(i k_1)(ik_2)}  \hat{w}_{k_1}\hat{w}_{k_2}\mathrm{e}^{i (k_1+k_2)x}\\
 &  \qquad + 2 \tau \hat{w}_0 \sum_{k_1 \in \mathbb{Z}} \hat{w}_{k_1} \mathrm{e}^{i k_1 x} - \tau \hat{w}_0^2 \\
  & =\p   \frac{i}{2} \mathrm{e}^{\m i t_n\partial_x^2}\Big[\mathrm{e}^{\m i  \tau\partial_x^2 } \left( \mathrm{e}^{\p i (t_n+\tau) \partial_x^2 } \partial_x^{-1} w\right) ^2 -\left( \mathrm{e}^{\p i t_n \partial_x^2 }\partial_x^{-1} w\right)^2\Big] + 2\tau \hat{w}_0 w- \tau \hat{w}_0^2.
\end{align*}
\endgroup
Together with the approximation in \eqref{approxv1} this yields the following integration scheme:
\begin{equation}\label{scheme1v}
\begin{aligned}
v^{n+1} &= \Phi_{t_n}^\tau(v^n)\\&= \left(1 - 2i \mu \tau \hat{v}_0^n\right) v^n + i\mu \tau( \hat{v}_0^n)^2 \\&\qquad\qquad \pp \frac{\mu}{2} \mathrm{e}^{\m i t_n\partial_x^2}\Big[\mathrm{e}^{\m i  \tau\partial_x^2 } \left( \mathrm{e}^{\p i  (t_n+\tau)\partial_x^2} \partial_x^{-1} v^n\right) ^2 - \left( \mathrm{e}^{\p i t_n \partial_x^2}\partial_x^{-1} v^n\right)^2\Big].
\end{aligned}
\end{equation}
In order to obtain an approximation to the original solution $u(t)$ of the quadratic Schr\"odinger equation \eqref{eq:qnls1}  at time $t_n = n\tau$ we twist the variable back again, i.e., we set $u^{n} := \mathrm{e}^{\p it_n \partial_x^2 } v^n$. This yields the following exponential-type integrator
\begin{equation}\label{scheme1u}
u^{n+1} =  \left(1 - 2i \mu \tau \hat{u}_0^n\right) \mathrm{e}^{\p i\tau \partial_x^2} u^n+ i\mu \tau( \hat{u}_0^n)^2 \pp  \frac{\mu}{2} \left( \mathrm{e}^{\p i \tau \partial_x^2 } \partial_x^{-1} u^n\right) ^2 \m \frac{\mu}{2}  \mathrm{e}^{\p i \tau \partial_x^2}\left(\partial_x^{-1} u^n\right)^2.
\end{equation}

\subsection{Quadratic Schr\"odinger equations of type (\ref{eq:qnls2})}  \label{secQNLS2}
For the quadratic Schr\"odinger equation \eqref{eq:qnls2} the equation in the twisted variable $v(t) = \mathrm{e}^{\m it \partial_x^2 } u(t)$ reads
\begin{equation*}\label{eq:v2}
i \partial_t v = \mu  \mathrm{e}^{\m it \partial_x^2 }  \left\vert \mathrm{e}^{\p it \partial_x^2 } v(t)\right\vert^2.
\end{equation*}
To construct our numerical scheme we proceed as above. We start from the following approximation
\begin{equation}
\begin{aligned}\label{eq:vdnls2}
v(t_n+\tau)
& \approx  v(t_n) - i \mu \int_0^\tau \mathrm{e}^{\m i(t_n+s) \partial_x^2 }  \left\vert \mathrm{e}^{\p i(t_n+s) \partial_x^2 } v(t_n)\right\vert^2\mathrm{d}s
\end{aligned}
\end{equation}
and rewrite the above integral in Fourier space as follows
\begin{equation*}
\begin{aligned}
 I^\tau(w,t_n)& =  \int_0^\tau \mathrm{e}^{\m i(t_n+s) \partial_x^2 }  \left\vert  \mathrm{e}^{\p i(t_n+s) \partial_x^2 } w\right \vert^2\mathrm{d}s\\
 & = \int_0^\tau \sum_{k_1,k_2} \mathrm{e}^{ i (t_n+s)\big[\p (k_1-k_2)^2 \m k_1^2 \pp k_2^2\big]}  \hat{w}_{k_1} \overline{\hat{w}}_{k_2}\mathrm{e}^{i (k_1-k_2)x}\mathrm{d}s.
 \end{aligned}
 \end{equation*}
The key-relation \eqref{eq:keyd2} now shows that
 \begin{equation*}
 \begin{aligned}
  I^\tau(w,t_n) & =  \sum_{\substack{k_1,k_2 \\ k_1\neq k_2, k_2 \neq 0}} \mathrm{e}^{ i t_n\big[\p (k_1-k_2)^2 \m k_1^2 \pp k_2^2\big]} \frac{\mathrm{e}^{ i \tau)\big[\p (k_1-k_2)^2 \m k_1^2 \pp k_2^2\big]}-1}{\m 2i k_2 (k_1-k_2)}  \hat{w}_{k_1}\overline{\hat{w}}_{k_2} \mathrm{e}^{i (k_1-k_2)x}\\
 &\qquad +  \tau \overline{\hat{w}}_0 \sum_{k_1} \hat{w}_{k_1} \mathrm{e}^{i k_1 x} + \tau \Vert w\Vert_{L^2}^2 \\
  & =\p\frac{i}{2} \partial_x^{-1} \mathrm{e}^{\m i (t_n+\tau)\partial_x^2 } \Big[ \left( \mathrm{e}^{\p i(t_n+\tau) \partial_x^2 }w\right) \left( \mathrm{e}^{\m  i (t_n+\tau)\partial_x^2} \partial_x^{-1} \overline{w}\right) \Big]\\
  &\quad \m  \frac{i}{2} \partial_x^{-1}\mathrm{e}^{\m it_n \partial_x^2} \Big[ \left( \mathrm{e}^{\p it_n \partial_x^2 } w\right)\left( \mathrm{e}^{ \m it_n \partial_x^2}\partial_x^{-1} \overline{w}\right) \Big]  + \tau \overline{\hat{w}}_0w + \tau \Vert w\Vert_{L^2}^2.
\end{aligned}
\end{equation*}
Together with the approximation in \eqref{eq:vdnls2} this yields (by twisting the variable back again) the following integration scheme in the original solution $u(t) =  \mathrm{e}^{\p it_n \partial_x^2 } v(t)$
\begin{equation}\label{scheme2u}
\begin{aligned}
u^{n+1} & = \left(1 - i \mu \tau \overline{\hat{u}}_0^n\right) \mathrm{e}^{\p i \tau \partial_x^2}u^n -  i\mu \tau \Vert u^n\Vert_{L^2}^2\\&\qquad  \pp \frac{\mu}{2} \partial_x^{-1} \Big[ \left( \mathrm{e}^{\p i  \tau\partial_x^2 }u^n\right) \left( \mathrm{e}^{\m i\tau  \partial_x^2} \partial_x^{-1} \overline{u}^n\right)- \mathrm{e}^{\p i \tau \partial_x^2 } \left(u^n \partial_x^{-1} \overline{u}^n\right) \Big].
\end{aligned}
\end{equation}

\subsection{Error analysis}\label{sec:Qlocerr}
In this section we carry out the error analysis of the exponential-type integrators \eqref{scheme1u} and \eqref{scheme2u}. In the following let $r>1/2$.

We commence with the quadratic Schr\"odinger equation of type \eqref{eq:qnls1}. Let $\phi^t$ denote the exact flow of~\eqref{eq:v1}, i.e., $v(t) = \phi^t(v(0))$. The following lemmata provide a local error bound of the scheme \eqref{scheme1v}.
\begin{lemma}\label{lem:loc1}
Let $r>1/2$ and assume that $v(t_k+t) = \phi^t(v(t_k)) \in H^r$ for $0 \leq t \leq \tau$, where $\phi^t$ denotes the exact flow of \eqref{eq:qnls1}. Then
\begin{equation*}
\Vert \phi^\tau(v(t_k)) - \Phi_{t_k}^\tau(v(t_k))\Vert_r \leq c \tau^2,
\end{equation*}
where $c$ depends on $\sup_{0\leq t \leq \tau} \Vert \phi^t(v(t_k))\Vert_r$.
\end{lemma}
\begin{proof}
As $\mathrm{e}^{i t \partial_x^2}$ is a linear isometry on $H^r$ for all $t \in \mathbb{R}$ we obtain by taking the difference of~\eqref{eq:vdnls} with the approximation \eqref{approxv1} and using the bilinear estimate \eqref{bil} that
\begin{equation*}
\begin{aligned}
&\Vert \phi^\tau(v(t_k)) - \Phi_{t_k}^\tau(v(t_k))\Vert_r  \leq  | \mu | \int_0^\tau \Big \Vert \left( \mathrm{e}^{\p i  (t_k+s)\partial_x^2} v(t_k+s)\right)^2 -  \left(\mathrm{e}^{\p i  (t_k+s)\partial_x^2} v(t_k)\right)^2 \Big\Vert_r \mathrm{d}s\\
&\qquad \leq c |\mu| \tau \sup_{0 \leq s \leq \tau} \left(\Vert  \mathrm{e}^{\p i  (t_k+s)\partial_x^2}\left( v(t_k+s) - v(t_k)\right)\Vert_r\Vert  \mathrm{e}^{\p i  (t_k+s)\partial_x^2}\left( v(t_k+s) + v(t_k)\right)\Vert_r\right)\\
& \qquad \leq c |\mu| \tau \sup_{0 \leq s \leq \tau} \Vert  v(t_k+s) - v(t_k)\Vert_r\\
& \qquad \leq c |\mu|^2 \tau \sup_{0 \leq s \leq \tau} \int_0^s \Big\Vert  \left( \mathrm{e}^{\p i  (t_k+\xi)\partial_x^2 } v(t_k+\xi)\right)^2\Big\Vert_r \mathrm{d}\xi,
\end{aligned}
\end{equation*}
where $c$ depends on $\sup_{0\leq s \leq \tau} \Vert \phi^s(v(t_k))\Vert_r$. This yields the assertion.
\end{proof}
Next we state the stability estimate.
\begin{lemma}\label{lem:stab1}
Let $r>1/2$ and $f,g \in H^r$. Then, for all $t \in \mathbb{R}$ we have
\begin{equation*}
\Vert \Phi_t^\tau(f) - \Phi_t^\tau(g)\Vert_r \leq \left(1+\tau \vert \mu\vert L\right) \Vert f - g \Vert_r,
\end{equation*}
where $L$ depends on $\Vert f+ g \Vert_r$.
\end{lemma}
\begin{proof}
The assertion follows from the representation of the numerical flow given in \eqref{approxv1} together with the bilinear estimate \eqref{bil}.
\end{proof}
The above lemmata allow us to prove the following global convergence result.
\begin{theorem}\label{thm1}
Let $r>1/2$. Assume that the exact solution of \eqref{eq:v1} satisfies $v(t) \in H^{r}$ for $0 \leq t \leq T$. Then, there exists a  constant   $\tau_0>0$ such that for all  step sizes   $\tau \leq \tau_0$ and $t_n \leq T$ we have that the global error of \eqref{scheme1v} is bounded by
\begin{equation*}
\Vert v(t_{n}) - v^{n} \Vert_r \leq c \tau,
\end{equation*}
where $c$ depends on $\sup_{0\leq t \leq T} \Vert v(t)\Vert_r$.
\end{theorem}
\begin{proof}
The proof follows the line of argumentation to the proof of Theorem \ref{thm1P} thanks to the local error bound in Lemma \ref{lem:loc1} and the stability estimate of Lemma \ref{lem:stab1}.
\end{proof}
The above theorem immediately yields first-order convergence of the exponential-type integrator~\eqref{scheme1u} (in $H^r$) without any additional smoothness assumptions on the exact solution of~\eqref{eq:qnls1}.
\begin{corollary}\label{coru1}
Let $r>1/2$. Assume that the exact solution of \eqref{eq:qnls1} satisfies $v(t) \in H^{r}$ for $0 \leq t \leq T$. Then, there exists a  constant   $\tau_0>0$ such that for all  step sizes   $\tau \leq \tau_0$ and $t_n \leq T$ we have that the global error of \eqref{scheme1u} is bounded by
\begin{equation*}
\Vert u(t_{n}) - u^{n} \Vert_r \leq c \tau,
\end{equation*}
where $c$ depends on $\sup_{0\leq t \leq T} \Vert u(t)\Vert_r$.
\end{corollary}
\begin{proof}
The bound follows directly from Theorem~\ref{thm1} as $\mathrm{e}^{i t \partial_x^2}$ is a linear isometry on $H^r$ for all $t \in \mathbb{R}$.
\end{proof}
Similarly to above we obtain the following first-order convergence result for the exponential-type integrator \eqref{scheme2u} approximating the solution of the quadratic Schr\"odinger equation \eqref{eq:qnls2}.
\begin{corollary}\label{cor2}
Let $r>1/2$. Assume that the exact solution of \eqref{eq:qnls2} satisfies $v(t) \in H^{r}$ for $0 \leq t \leq T$. Then, there exists a  constant   $\tau_0>0$ such that for all  step sizes   $\tau \leq \tau_0$ and $t_n \leq T$ we have that the global error of \eqref{scheme2u} is bounded by
\begin{equation*}
\Vert u(t_{n}) - u^{n} \Vert_r \leq c \tau,
\end{equation*}
where $c$ depends on $\sup_{0\leq t \leq T} \Vert u(t)\Vert_r$.
\end{corollary}
\begin{proof}
The  proof is similar to that of  Corollary \ref{coru1} and will therefore be omitted here.
\end{proof}

\section{Numerical experiments}\label{sec:num}
In this section we compare the newly developed exponential-type integrator \eqref{schemePu} applied to Schr\"odinger equations with low regularity solutions in the energy space with the classical Lie splitting, the classical Strang splitting and the classical first-order exponential integrator. In the numerical experiments we use a standard Fourier pseudospectral method for the space discretization where we choose the largest Fourier mode $K = 2^{10}$ (i.e., the spatial mesh size $\Delta x =  0.0061$). The numerical experiments show that the new integrator \eqref{schemePu} is preferable over the classical splitting and exponential integration methods for \emph{low regularity solutions}. In particular, the experiments indicate that the proposed integrator \eqref{schemePu} is convergent of order one even for rough solutions whereas the classical methods suffer from order reduction.

\blue Although the new integrator \eqref{schemePu} was derived for nonlinearities with integer $p\ge 1$, it can be applied to problem \eqref{nls} with non-integer $p$ as well. Preliminary numerical results for small non-integer values of $p$ are presented in Section~\ref{sec:numP}.\black

\subsection{Numerical experiments for  Schr\"odinger equations of type (\ref{nls})}
In this section we consider nonlinear Schr\"odinger equations of type \eqref{nls}.  The associated Lie splitting (of \emph{classical order one}) reads
\begin{equation}\label{lieP}
\begin{aligned}
& u^{n+1/2}_L = \mathrm{e}^{\p i \tau \Delta} u^n_L,\\
& u^{n+1}_L=  \mathrm{e}^{-i \tau\mu\big \vert u_L^{n+1/2}\big\vert^{2p}}   u_L^{n+1/2},
\end{aligned}
\end{equation}
and the associated Strang splitting (of \emph{classical order two}) is given by
\begin{equation}\label{strangP}
\begin{aligned}
& u^{n+1/2}_{-} = \mathrm{e}^{\p i \frac{\tau}{2} \Delta} u^n_S,\\
& u^{n+1/2}_{+} =  \mathrm{e}^{-i \tau \mu\big\vert u^{n+1/2}_{-}\big\vert^{2p}}   u^{n+1/2}_{-},
\\
& u^{n+1}_S = \mathrm{e}^{\p i \frac{\tau}{2} \Delta}u^{n+1/2}_{+}.
\end{aligned}
\end{equation}
Furthermore, the classical first-order exponential integrator (of \emph{classical order one})  is defined through
\begin{equation}\label{expiP}
u^{n+1}_E = \mathrm{e}^{\p i \tau \Delta} u^n_E - i  \mu \tau  \varphi_1(\p i \tau \Delta) \Big(\vert u_E^n\vert^{2p} u_E^n\Big).
\end{equation}
In our numerical experiments we  choose $\mu=1$ (unless otherwise noted) and   set 
\begin{equation}\label{in:rand}
  \mathcal{U}^K:= [u(0,x_1), u(0,x_2),\ldots,u(0,x_{2K})] =  \mathrm{rand}(2K,1)+i\, \mathrm{rand}(2K,1) \in \mathbb{C}^{2K},
\end{equation}
where   $x_j = (-1+ \frac{j}{K})\pi$ and   $\mathrm{rand}(2K,1)$ returns $2K$ uniformly distributed random numbers between $0$ and $1$. We compare the above schemes with the newly derived exponential-type integrator \eqref{schemePu} for initial values
\begin{equation}\label{inv}
u^K_\vartheta = \vert \partial_{x,K}\vert ^{-\vartheta} \mathcal{U}^K, \quad\big (\vert \partial_{x,K}\vert^{-\vartheta}\big)_k :=
\left\{
\begin{array}{ll}
|k|^{-\vartheta} &\mbox{if} \quad k \neq 0,\\
0 & \mbox{if}\quad k = 0
\end{array}
\right.
\end{equation}
for different values of $\vartheta$ normalized in $L^2$.   (For typical initial values, see Figure \ref{fig:ini}.)  All experiments are carried out with constant step sizes $\tau = \frac{j}{512}$ for $j=1,\ldots, 512$.

\begin{figure}[h!]
\centering
\includegraphics[width=0.43\linewidth]{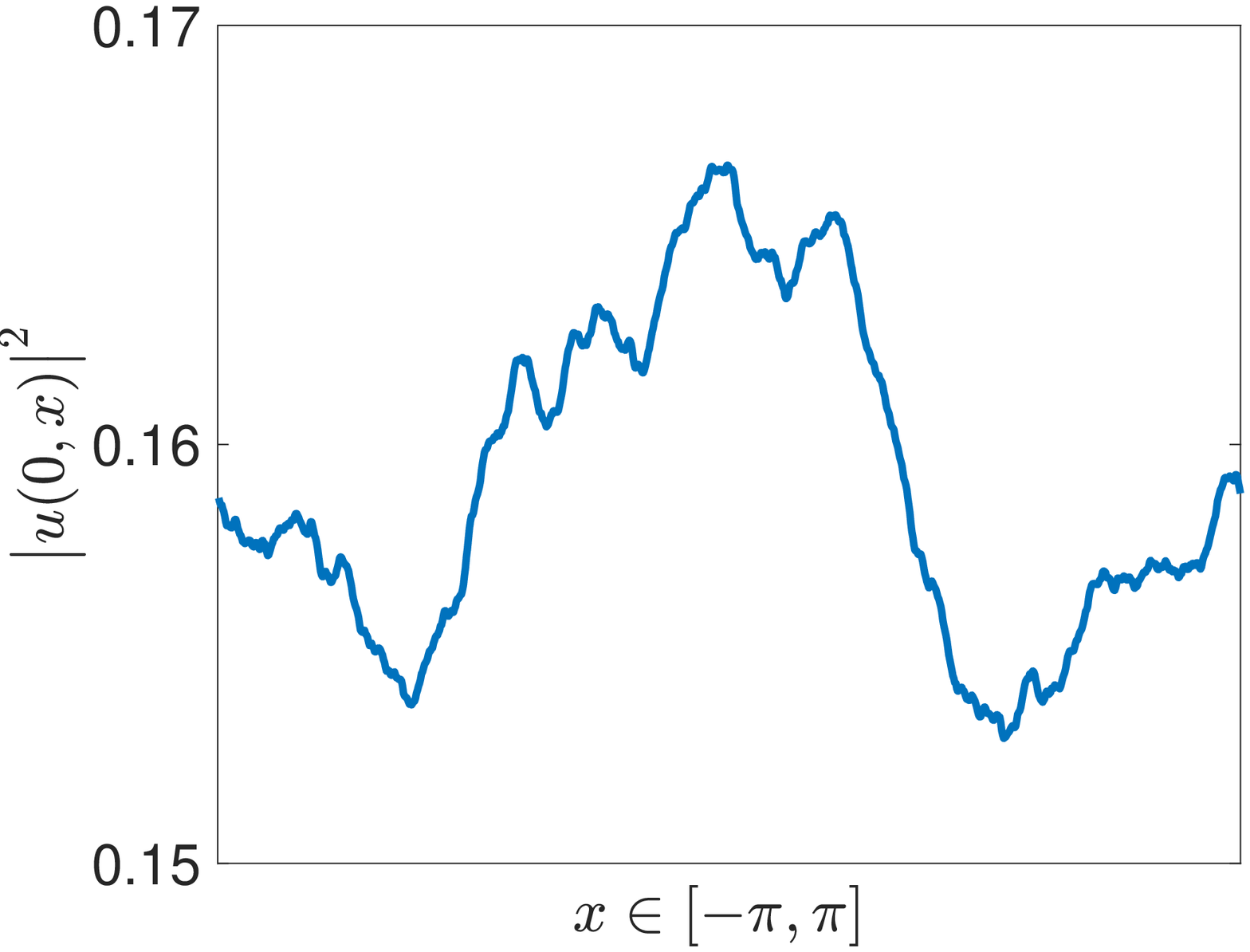}
\hskip1cm
\includegraphics[width=0.43\linewidth]{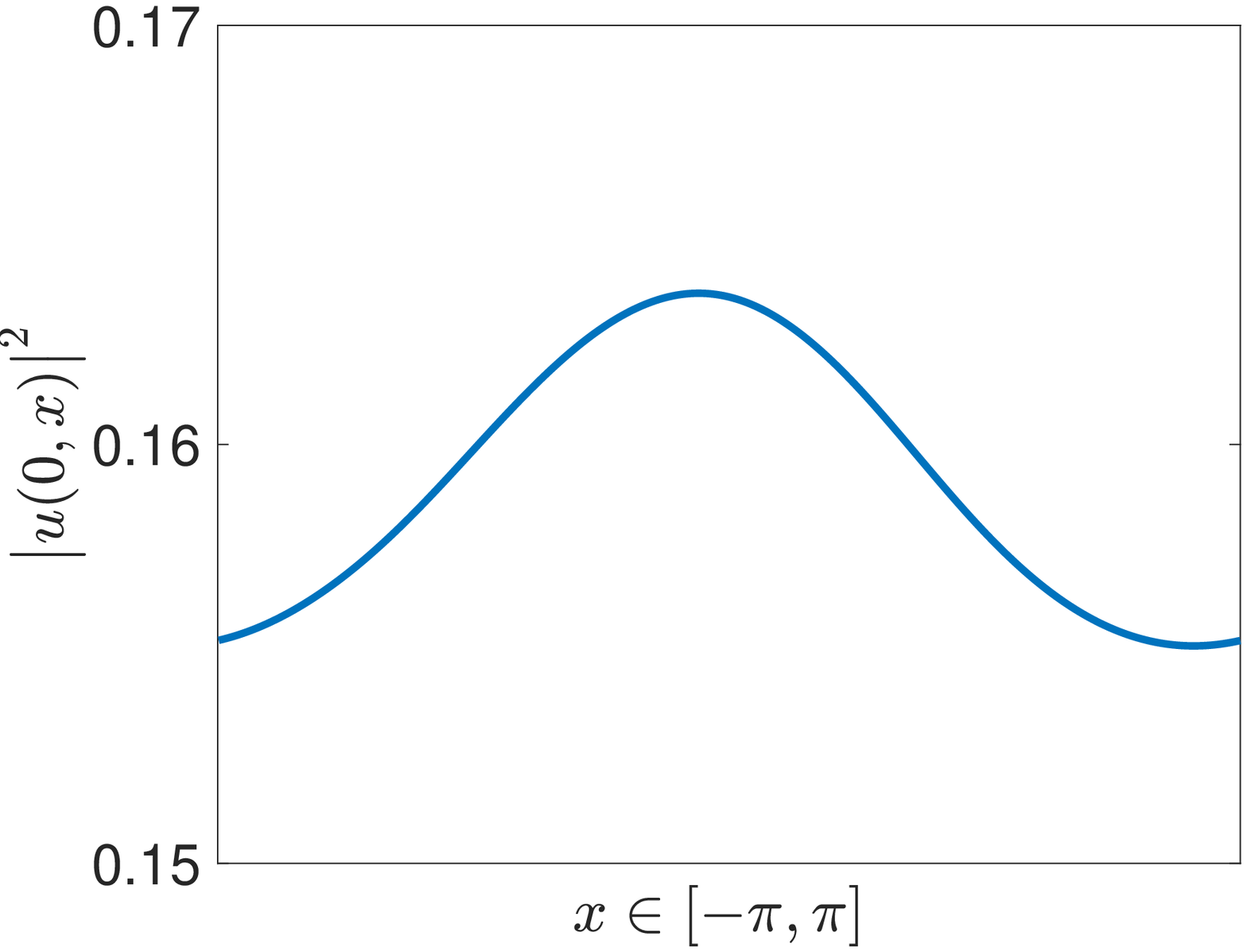}
\caption{Initial values \eqref{inv} normalized in $L^2$ for two different values of $\vartheta$. Left: $\vartheta = \frac{3}{2}$. Right: $\vartheta = 5$. }\label{fig:ini}
\end{figure}

\subsubsection{Cubic Schr\"odinger equations}
In this section we consider the classical cubic Schr\"odinger equation, i.e., $p = 1$ in \eqref{nls}. The error   at time $T=1$   measured in a discrete $H^1$ norm of the exponential-type integrator \eqref{schemePu}, the Lie splitting scheme \eqref{lieP}, the Strang splitting scheme \eqref{strangP} as well as the classical exponential integrator \eqref{expiP} for the initial value \eqref{inv} normalized in $L^2$ is illustrated in Figure \ref{fig:randCub2} for $\vartheta = 3/2$, respectively, $\vartheta = 2$ and in Figure \ref{fig:randCub4} for $\vartheta = 3$, respectively, $\vartheta = 5$.  As a reference solution we take the schemes themselves with a very small time step size if $\vartheta \neq 5$. For $\vartheta = 5$ we take the Strang splitting scheme \eqref{strangP} with a very small time step size as the reference solution for all schemes.
The experiments show that Strang splitting is second-order convergent for $\vartheta\ge 5$, and that Lie splitting is first-order convergent for $\vartheta\ge 3$. These observations are in line with the convergence proofs from the literature \cite{ESS16, Lubich08}, see also our discussion in the introduction. For smaller values of $\vartheta$, both Lie and Strang splitting show a zig-zag behaviour, depending on whether the local errors accumulate or happy error cancellation occurs. The classical exponential integrator of order one is first-order convergent for $\vartheta\ge 3$ as expected (see the discussion in the introduction). For smaller values of $\vartheta$ the convergence behaviour of the exponential integrator gets less regular and order reduction occurs. On the other hand, the new exponential-type integrator introduced in this paper is first-order convergent for all $\vartheta \ge 2$. This behaviour is in line with Theorem~\ref{thm1P}. The integrator even shows first-order convergence and small errors for $\vartheta=\frac32$.
\begin{figure}[h!]
\centering
\includegraphics[width=0.43\linewidth]{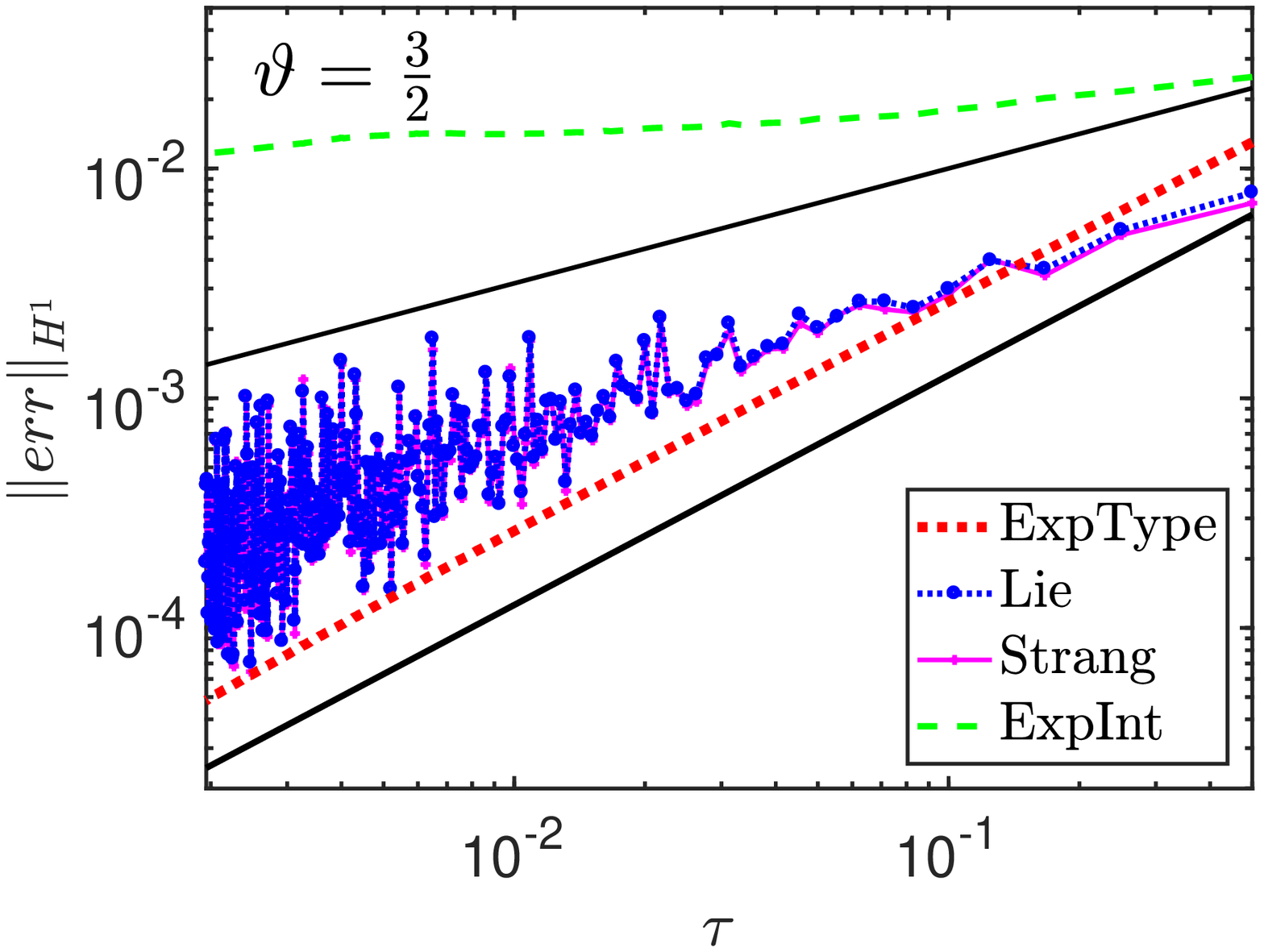}
\hfill
\includegraphics[width=0.43\linewidth]{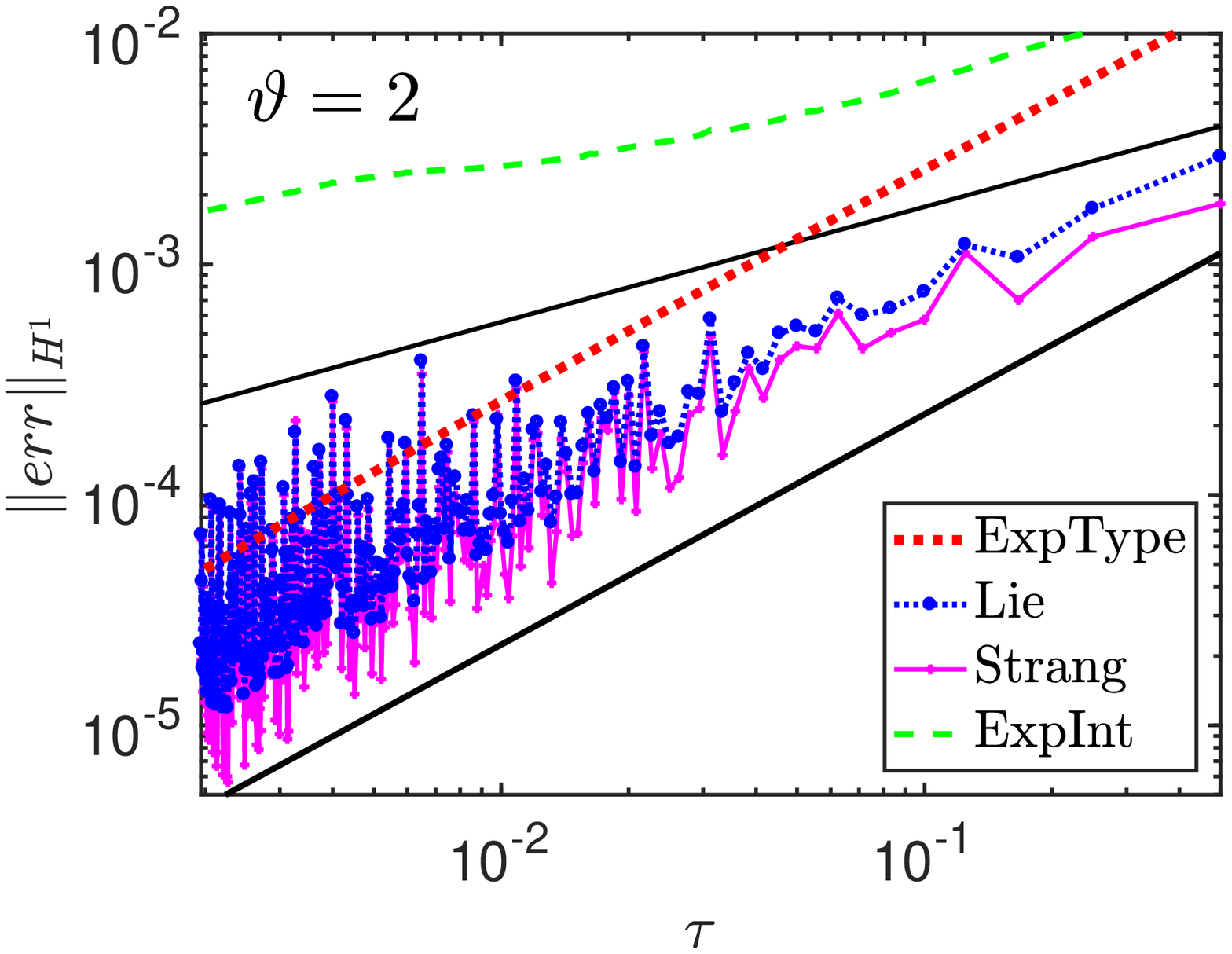}
\caption{Error in the energy space $H^1$ of the Lie splitting \eqref{lieP} (blue, dotted circle), Strang splitting \eqref{strangP} (magenta, star), classical exponential integrator \eqref{expiP}  (green, dashed) and exponential-type integration scheme \eqref{schemePu} (red, dash dotted). Left picture:  $H^{3/2}$ solutions. Right picture: $H^2$ solutions.  The slope of the continuous lines is one half and one, respectively.}\label{fig:randCub2}
\end{figure}

\begin{figure}[h!]
\centering
\includegraphics[width=0.43\linewidth]{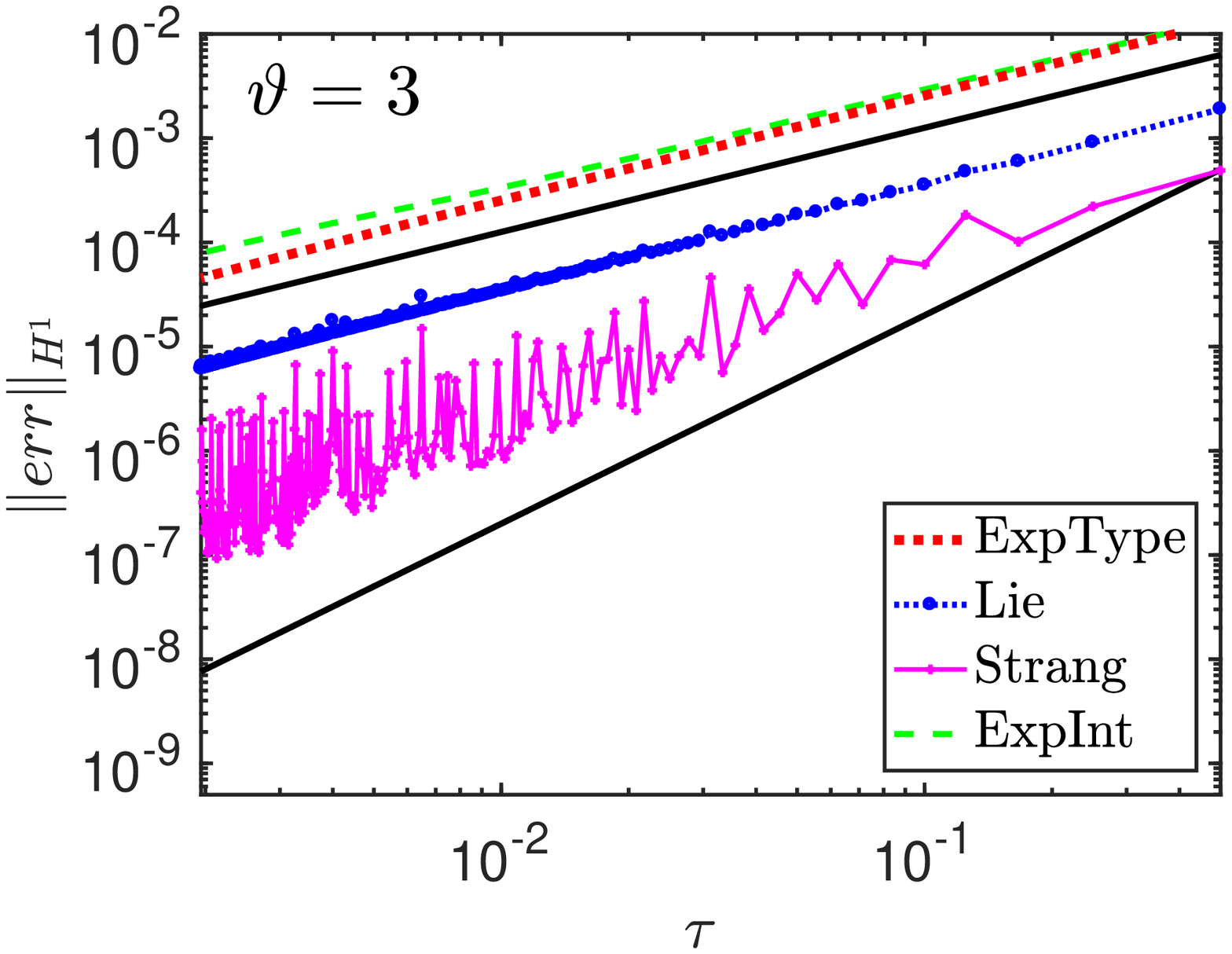}
\hfill
\includegraphics[width=0.43\linewidth]{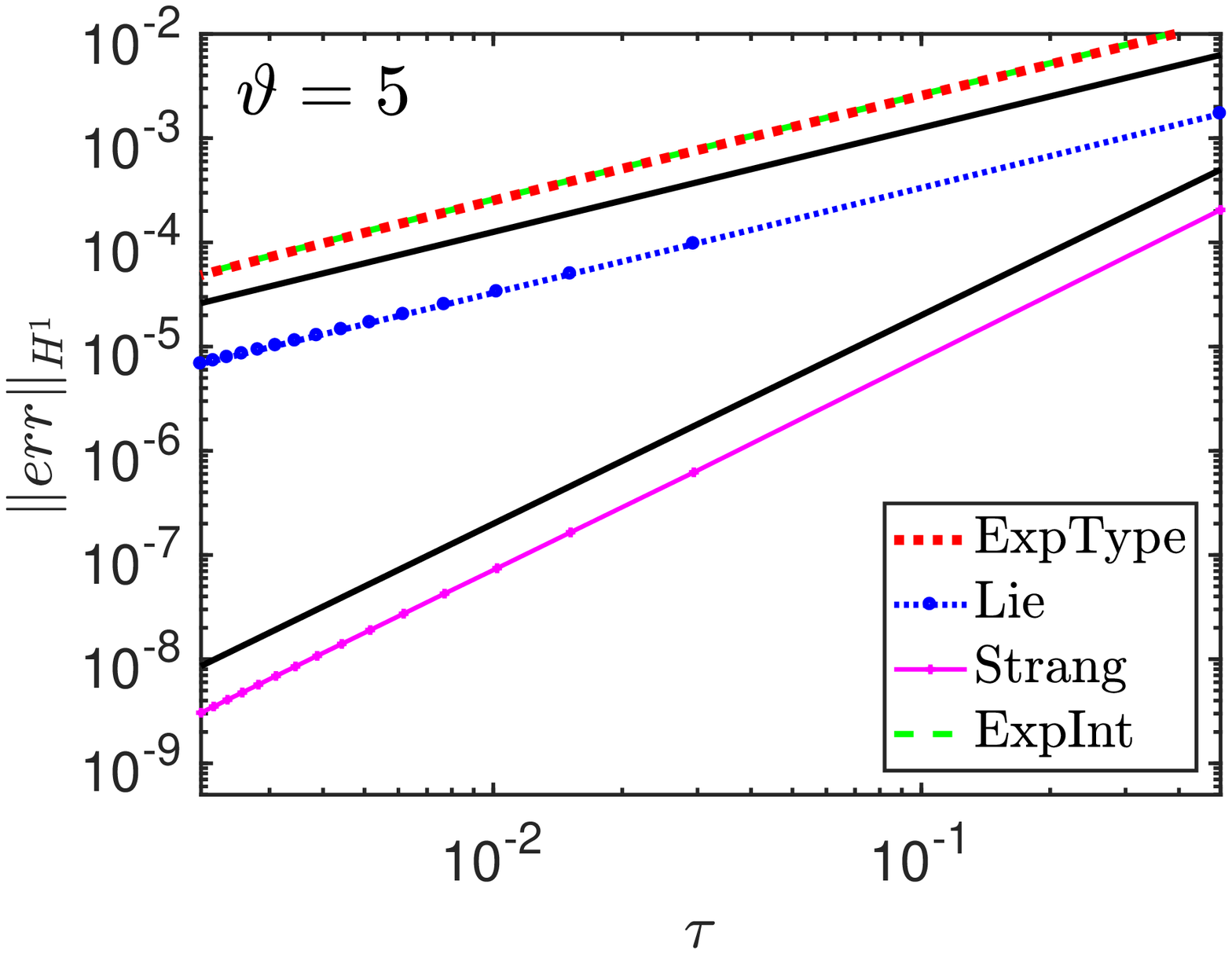}
\caption{Error in the energy space $H^1$ of the Lie splitting \eqref{lieP} (blue, dotted circle), Strang splitting \eqref{strangP} (magenta, star), classical exponential integrator \eqref{expiP}  (green, dashed) and exponential-type integration scheme \eqref{schemePu} (red, dash dotted).  Left picture:  $H^{3}$ solutions. Right picture: $H^5$ solutions. The slope of the continuous lines is one and two, respectively.}\label{fig:randCub4}
\end{figure}

\subsubsection{Quintic Schr\"odinger equations}

In this section we consider the quintic Schr\"odinger equation, i.e., $p = 2$ in \eqref{nls}, which appears for instance as the mean field limit of a Boson gas with three-body interactions, see \cite{CP11}. The error   at time $T=1$   measured in a discrete $H^1$ norm of the exponential-type integrator \eqref{schemePu}, the Lie splitting scheme \eqref{lieP}, the Strang splitting scheme \eqref{strangP} as well as the classical exponential integrator \eqref{expiP} for the initial value \eqref{inv}
normalized in $L^2$ is illustrated in Figure \ref{fig:randQuin1} for $\vartheta = 3/2$, respectively, $\vartheta = 2$ and in Figure \ref{fig:randQuin2} for $\vartheta = 3$, respectively, $\vartheta = 5$. As a reference solution we take the schemes themselves with a very small time step size if $\vartheta \neq 5$. For $\vartheta = 5$ we take the Strang splitting scheme \eqref{strangP} with a very small time step size as the reference solution for all schemes.  The outcome of these numerical experiments is almost the same as for the cubic nonlinear Schr\"odinger equation. The only notable difference is the exponential-type integrator for $\vartheta=\frac32$, where the convergence behaviour is now less regular. Note, however, that the errors are still much smaller compared to the other methods.

\begin{figure}[h!]
\centering
\includegraphics[width=0.43\linewidth]{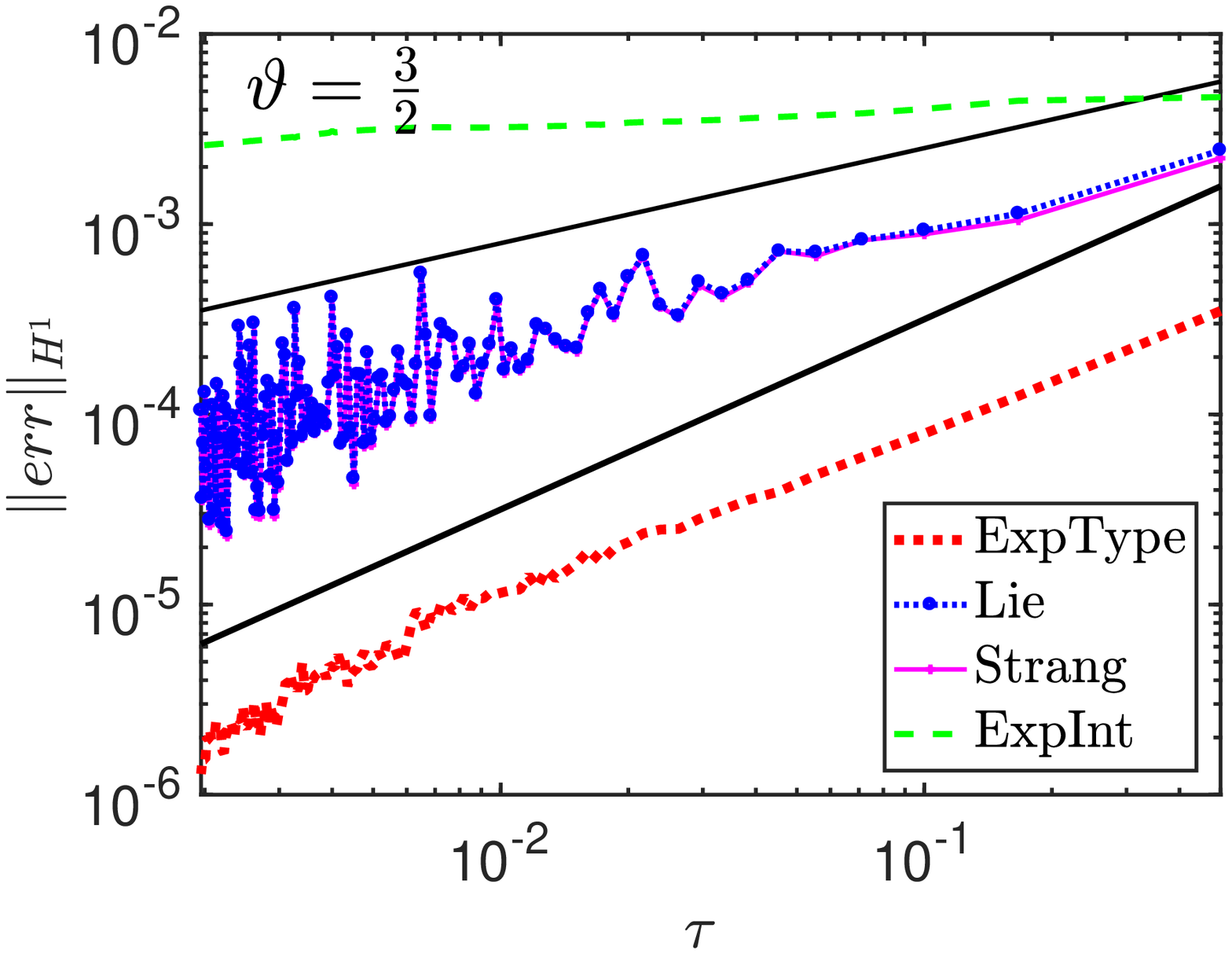}
\hfill
\includegraphics[width=0.43\linewidth]{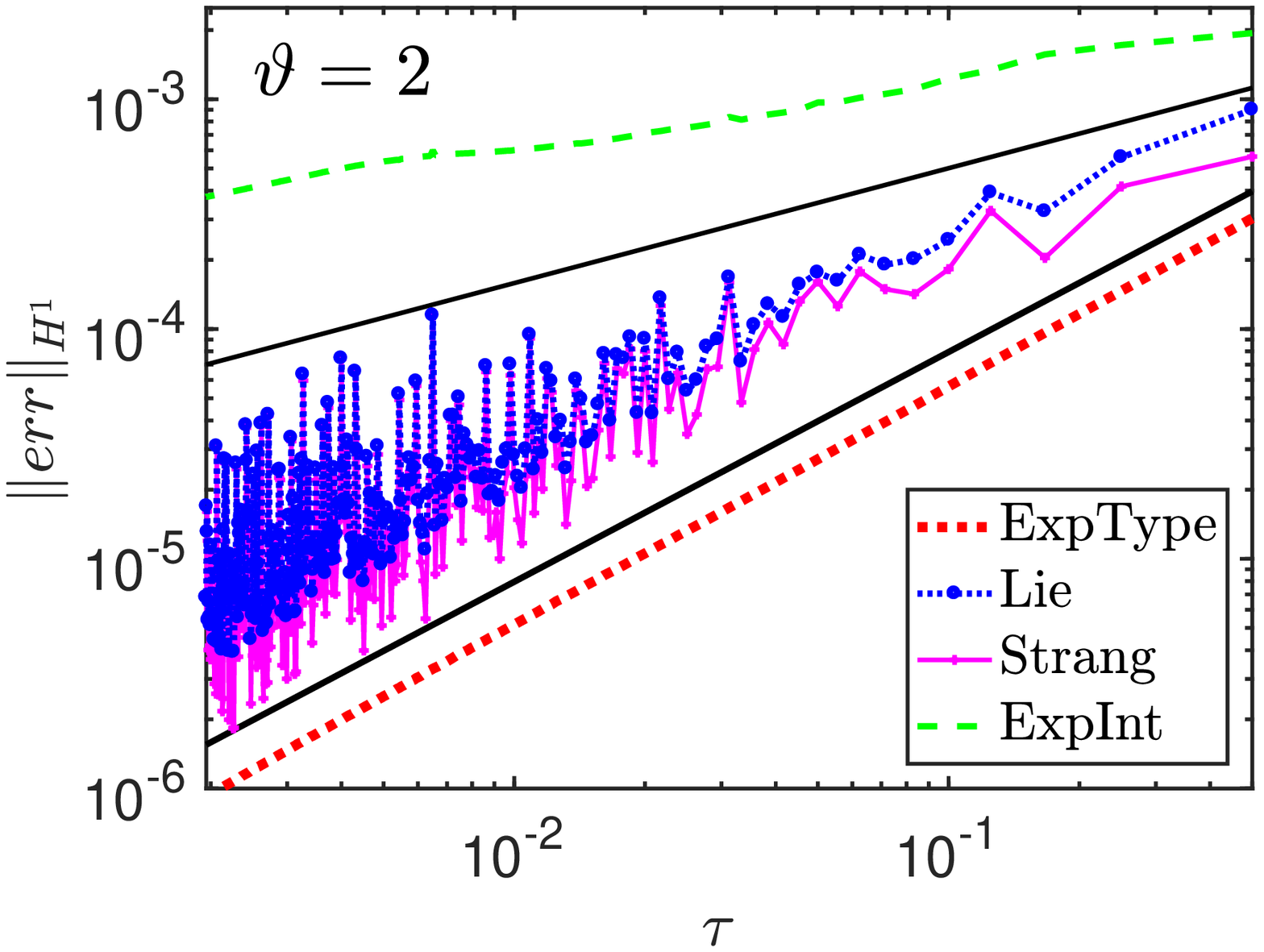}
\caption{Error in the energy space $H^1$ of the Lie splitting \eqref{lieP} (blue, dotted circle), Strang splitting \eqref{strangP} (magenta, star), classical exponential integrator \eqref{expiP}  (green, dashed) and exponential-type integration scheme \eqref{schemePu} (red, dash dotted). Left picture:  $H^{3/2}$ solutions. Right picture: $H^2$ solutions. Left picture: The slope of the continuous lines is one quarter and one, respectively. Right picture:  The slope of the continuous lines is one half and one, respectively. }\label{fig:randQuin1}
\end{figure}

\begin{figure}[h!]
\centering
\includegraphics[width=0.43\linewidth]{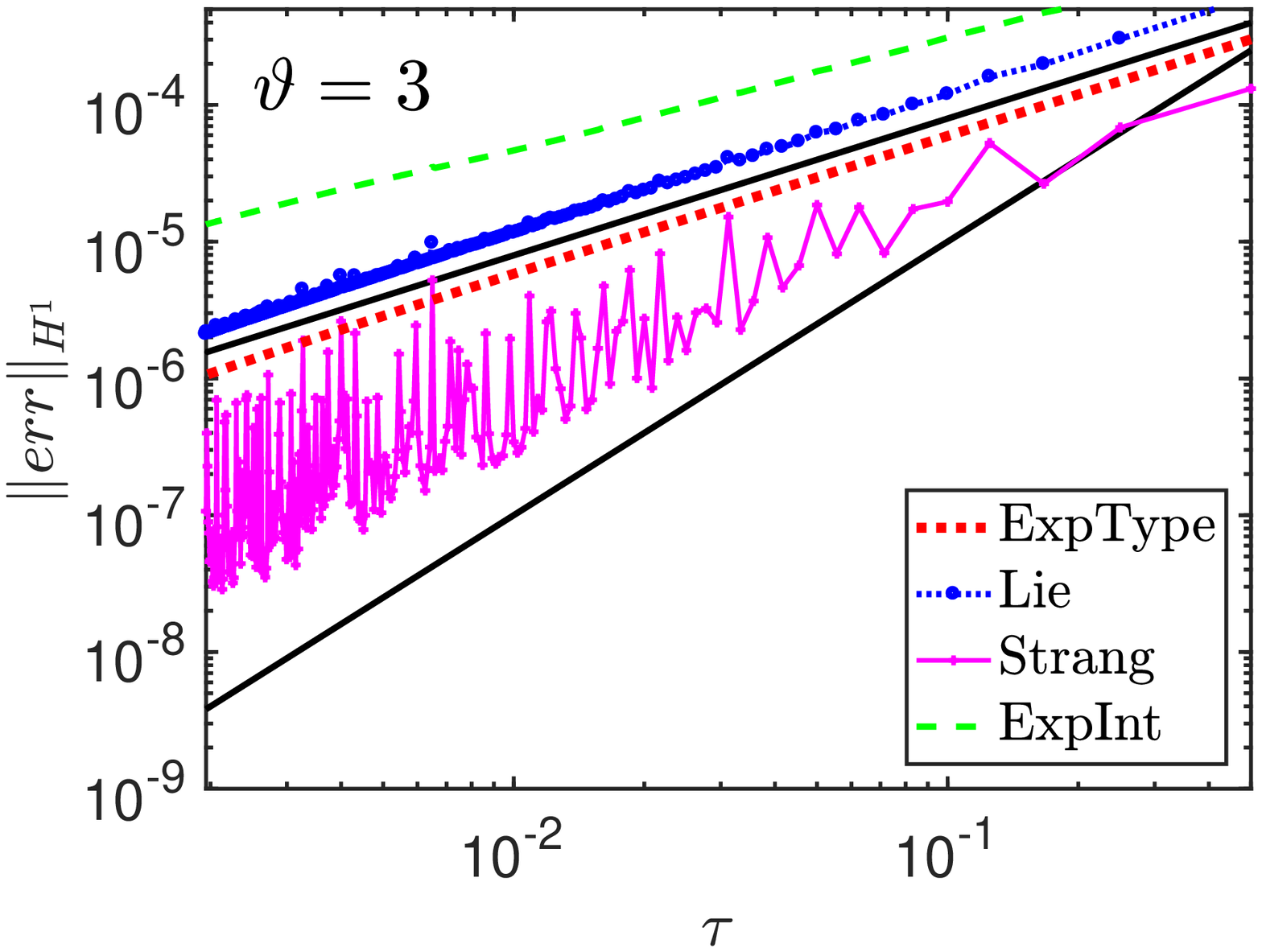}
\hfill
\includegraphics[width=0.43\linewidth]{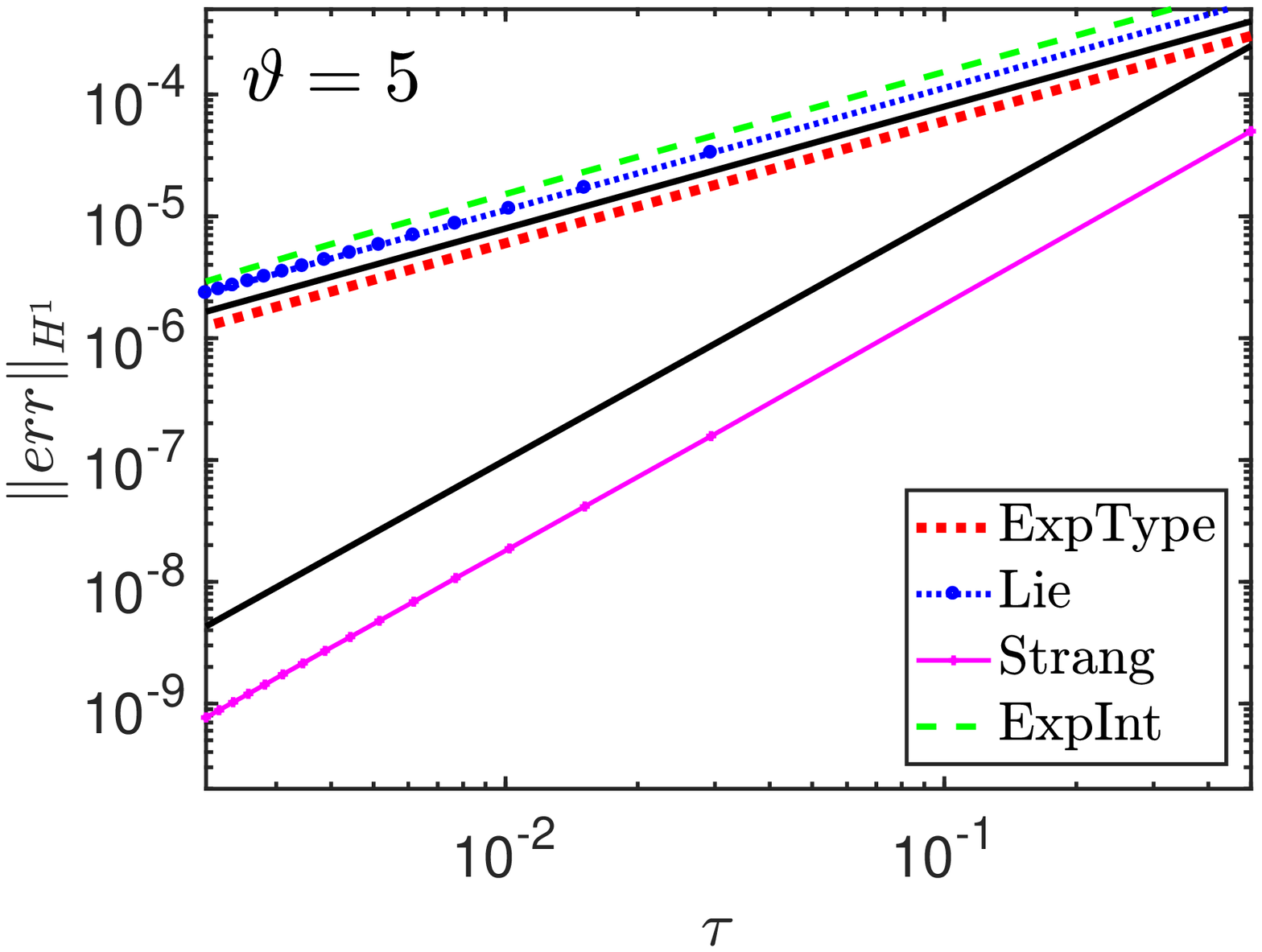}
\caption{Error in the energy space $H^1$ of the Lie splitting \eqref{lieP} (blue, dotted circle), Strang splitting \eqref{strangP} (magenta, star), classical exponential integrator \eqref{expiP}  (green, dashed) and exponential-type integration scheme \eqref{schemePu} (red, dash dotted). Left picture:  $H^{3}$ solutions. Right picture: $H^5$ solutions. The slope of the continuous lines is one and two, respectively. }\label{fig:randQuin2}
\end{figure}

\subsection{Numerical experiments for quadratic Schr\"odinger equations}

In this section we consider quadratic Schr\"odinger equations of type \eqref{eq:qnls1}. In order to derive the splitting methods, we split the right-hand side of \eqref{eq:qnls1} into the linear and nonlinear part, respectively. Note that (for sufficiently small $t$) the exact solution of the subproblem
$$
i u'(t) = \mu u(t)^2
$$
is given by
$$
u(t) = \frac{u(0)}{1+i \mu tu(0)}.
$$
Thus, the associated Lie splitting scheme (of \emph{classical order one}) reads
\begin{equation}\label{lie}
u_{L}^{n+1} = \mathrm{e}^{\p i \tau \partial_x^2}  \left(\frac{u_L^n}{1+i \mu \tau u_L^n}\right)
\end{equation}
and the associated Strang splitting scheme (of \emph{classical order two})  is given by
\begin{equation}\label{strang}
\begin{aligned}
u_{S}^{n+1/2} &= \mathrm{e}^{\p i \frac{\tau}{2} \partial_x^2}  u_S^n\\
u_{S}^{n+1} &= \mathrm{e}^{\p i \frac{\tau}{2} \partial_x^2}  \left(\frac{u_S^{n+1/2}}{1+i \mu \tau u_S^{n+1/2}}\right).
\end{aligned}
\end{equation}
In Examples \ref{ex:smooth} and \ref{ex:nsmooth}  we compare the above splitting methods as well as the classical first-order exponential integration scheme
\begin{equation}\label{expicl}
u^{n+1}_E = \mathrm{e}^{\p i \tau \partial_x^2} u^n_E - i  \mu \tau  \varphi_1(\p i \tau \partial_x^2) \big(u^n_E\big)^2
\end{equation}
with the newly derived exponential-type integrator \eqref{scheme1u} for \emph{smooth and non-smooth solutions}, respectively, where we set $\mu = 1$ and integrate up to $T = 1$. The numerical experiments in particular indicate that the exponential-type integrator \eqref{scheme1u} is preferable over the commonly used splitting methods for \emph{low regularity solutions}. This is due to the fact that within the exponential integration scheme all the \emph{stiff parts} (i.e., the terms involving the differential operator $\partial_x^2$) are solved exactly.

In Example \ref{ex:SNsmooth} we consider a quadratic Schr\"odinger equation of type \eqref{eq:qnls1} with a small nonlinearity on a longer time interval $T = 10$. The numerical experiment indicates the favorable behavior of the exponential integrator \eqref{scheme1u} even in the smooth case. This is due to the fact that its error constant is triggered by the nonlinearity (and not by the differential operator $\partial_x^2$).

\begin{example}[Smooth solutions]\label{ex:smooth}
We choose the smooth initial value
\begin{equation}\label{in:smooth}
u(0,x) = \sin x \cos x \in   \mathcal{C}^\infty(\mathbb{T})
\end{equation}
normalized in $L^2$. The error   at time $T=1$   measured in a discrete $L^2$ norm of the exponential-type integrator \eqref{scheme1u}, the Lie splitting scheme \eqref{lie}, the Strang splitting scheme \eqref{strang}, and   the classical exponential integrator   \eqref{expicl} is illustrated in Figure \ref{fig:smooth}. As a reference solution we take the Strang splitting scheme \eqref{strang} with a very small time step size.  All the methods show their classical orders of convergence, as expected.

\begin{figure}[h!]
\centering
\includegraphics[width=0.43\linewidth]{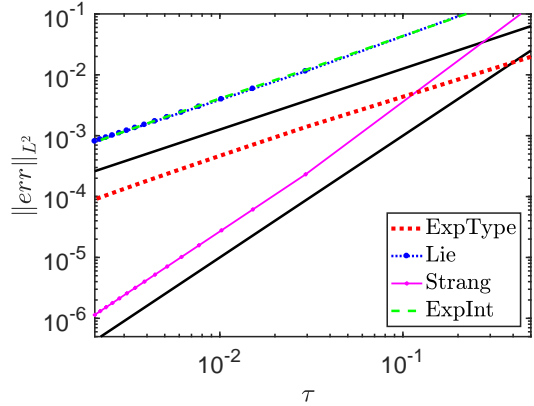}
\caption{Error of the Lie splitting \eqref{lie} (blue, dotted circle), Strang splitting \eqref{strang} (magenta, star), classical exponential integrator \eqref{expicl} (green, dashed) and exponential-type integration scheme \eqref{scheme1u} (red, dash dotted) for a \emph{smooth solution} with initial value \eqref{in:smooth} and $\mu=1$. The slope of continuous lines is one and two, respectively.}\label{fig:smooth}
\end{figure}
\end{example}

\begin{example}[Non-smooth solutions]\label{ex:nsmooth}
We   choose the non-smooth initial value \eqref{in:rand} normalized in $L^2$.   The error   at time $T=1$   measured in a discrete $L^2$ norm of the exponential-type integrator \eqref{scheme1u}, the Lie splitting scheme \eqref{lie}, the Strang splitting scheme \eqref{strang} and the classical exponential integrator \eqref{expicl} is illustrated in Figure \ref{fig:rand}. As a reference solution we take the schemes themselves with a very small time step size.  For this example, the classical exponential integrator is not convergent. Lie and Strang splitting suffer from a strong order reduction (down to order one half). Here, the new exponential-type integrator is by far the best method. It is first-order convergent, as predicted by Theorem \ref{thm1}.
\begin{figure}[h!]
\centering
\includegraphics[width=0.43\linewidth]{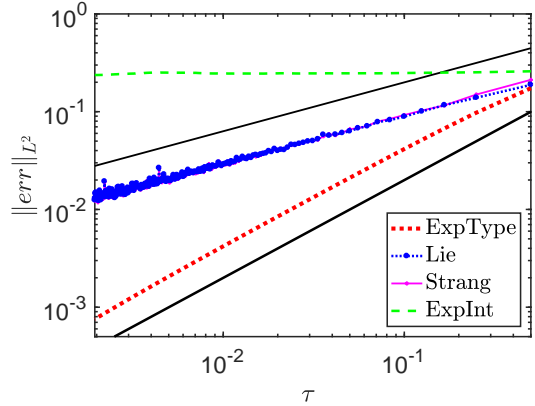}
\caption{Error of the Lie splitting \eqref{lie} (blue, circle), Strang splitting \eqref{strang} (magenta, star), classical exponential integrator \eqref{expicl} (green, diamond) and exponential-type integration scheme \eqref{scheme1u} (red, cross) for a \emph{non-smooth solution} with initial value \eqref{in:rand}  and $\mu=1$. The slope of continuous lines is one half and one, respectively.}\label{fig:rand}
\end{figure}
\end{example}

\begin{example}[Small nonlinearity]\label{ex:SNsmooth}
We consider the quadratic Schr\"odinger equation of type \eqref{eq:qnls1} with   small   nonlinearity, i.e.,
\begin{equation}\label{thelast}
i \partial_t u =\m \partial_x^2 u + \mu u^2 ,\quad \mu = 0.01, \quad T = 10.
\end{equation}
The error   at time $T=10$   measured in a discrete $L^2$ norm of the exponential-type integrator \eqref{scheme1u}, the Lie splitting scheme \eqref{lie}, the Strang splitting scheme \eqref{strang} and the classical exponential integrator  \eqref{expicl} for smooth and non-smooth solutions is illustrated in Figure \ref{fig:smallN} left and right, respectively. As a reference solution we take the Strang splitting scheme for smooth and the schemes themselves for non-smooth solutions with a very small time step size.
For smooth initial data, all methods show their classical orders of convergence. The new exponential-type integrator, however, is by far the most accurate scheme (for the considered range of step sizes). For non-smooth initial data, the classical exponential integrator does not converge; Lie and Strang splitting suffer again from a strong order reduction down to order one half. Only the new exponential-type integrator is very accurate and first-order convergent, as expected.
\begin{figure}[h!]
\centering
\includegraphics[width=0.43\linewidth]{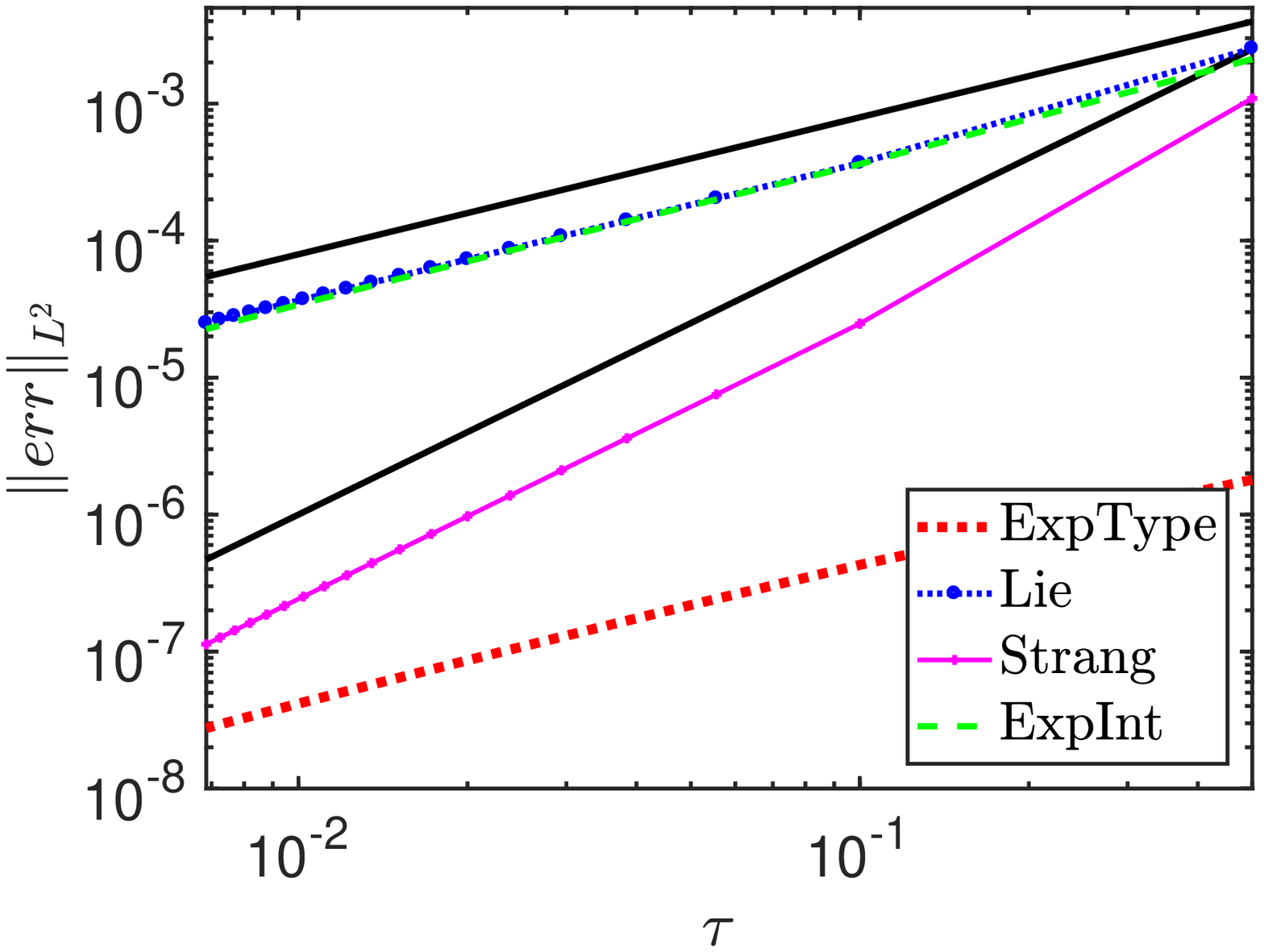}
\hfill
\includegraphics[width=0.43\linewidth]{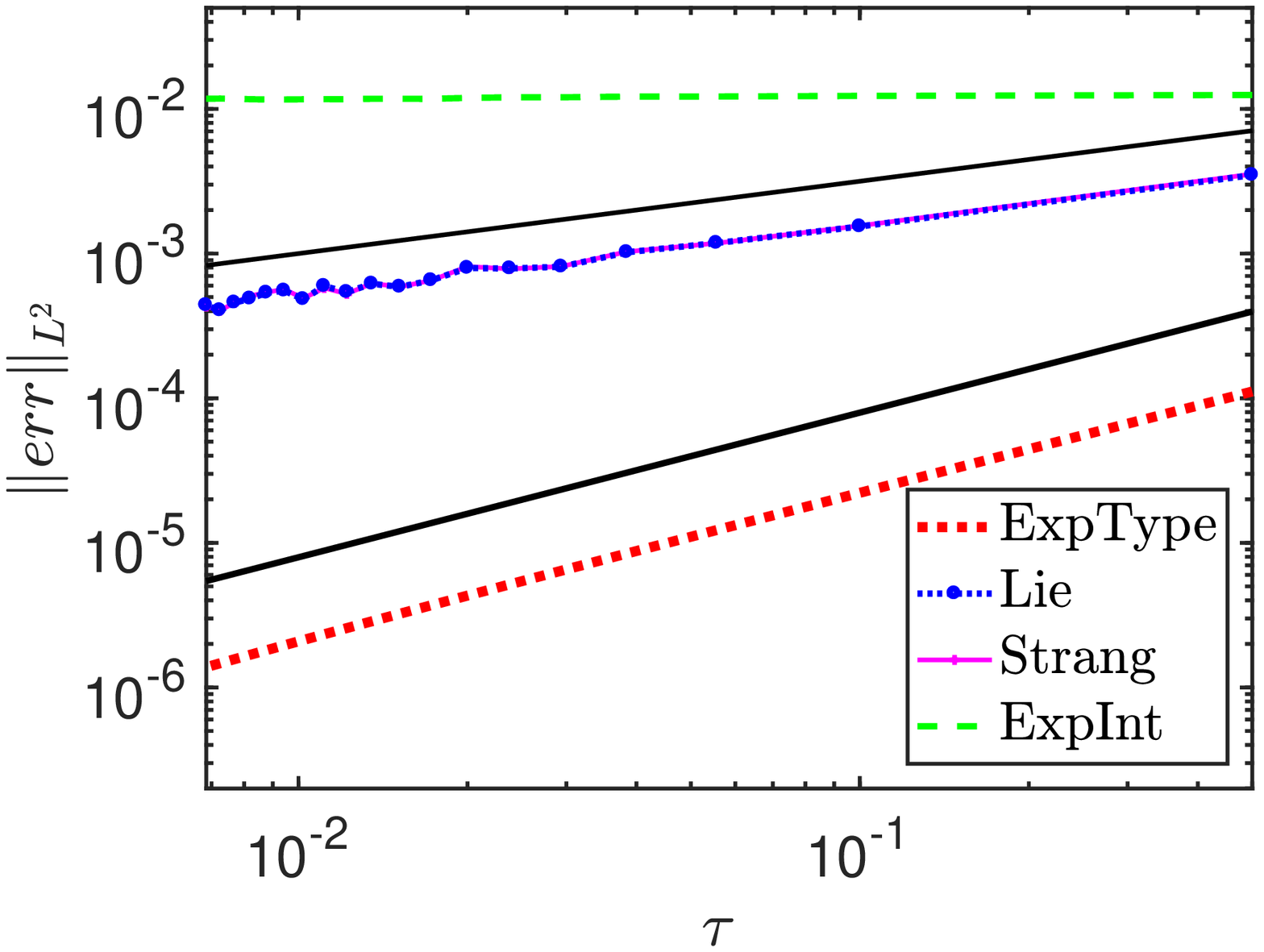}
\caption{Error of the Lie splitting \eqref{lie} (blue, dotted circle), Strang splitting \eqref{strang} (magenta, star), classical exponential integrator \eqref{expicl} (green, dashed) and exponential-type integration scheme \eqref{scheme1u} (red, dash dotted) for a \emph{small nonlinearity} \eqref{thelast}. Left picture: Smooth initial value \eqref{in:smooth}. The slope of the continuous lines is one and two, respectively.  Right picture: Non-smooth initial value \eqref{in:rand}. The slope of the continuous lines is one half, and one, respectively. }\label{fig:smallN}
\end{figure}
\end{example}

\subsection{Numerical experiments for semilinear Schr\"odinger equations with non-integer exponents $p$.}\label{sec:numP}

In this section we numerically investigate the convergence rates in the case of semilinear Schr\"odinger equations of type \eqref{nls} with non-integer exponents $p$. In Figure \ref{fig:pQuarter} we plot the error at time $T=1$ measured in a discrete $H^1$ norm  of the exponential-type integrator \eqref{schemePu}, the Lie splitting scheme \eqref{lieP}, the Strang splitting scheme \eqref{strangP} as well as the classical exponential integrator \eqref{expiP} for  different values of $p >0$. In all the simulations we take the smooth initial value
\begin{equation}\label{PiniD}
u(0,x) = \mathrm{sin}\, x \in \mathcal{C}^\infty(\mathbb{T})
\end{equation}
and $\mu=1$. Note, however, that the non-linearity is only smooth for $p> \frac{1}{2}$.

For $p=1$, all methods show their classical orders of convergence, as expected. For smaller values of $p$, however, the convergence behaviour of Strang and Lie splitting deteriorates. Surprisingly, both exponential integrators retain their first-order convergence down to very small values of $p$. This is the subject of future research.

\begin{figure}[h!]
\centering
\includegraphics[width=0.43\linewidth]{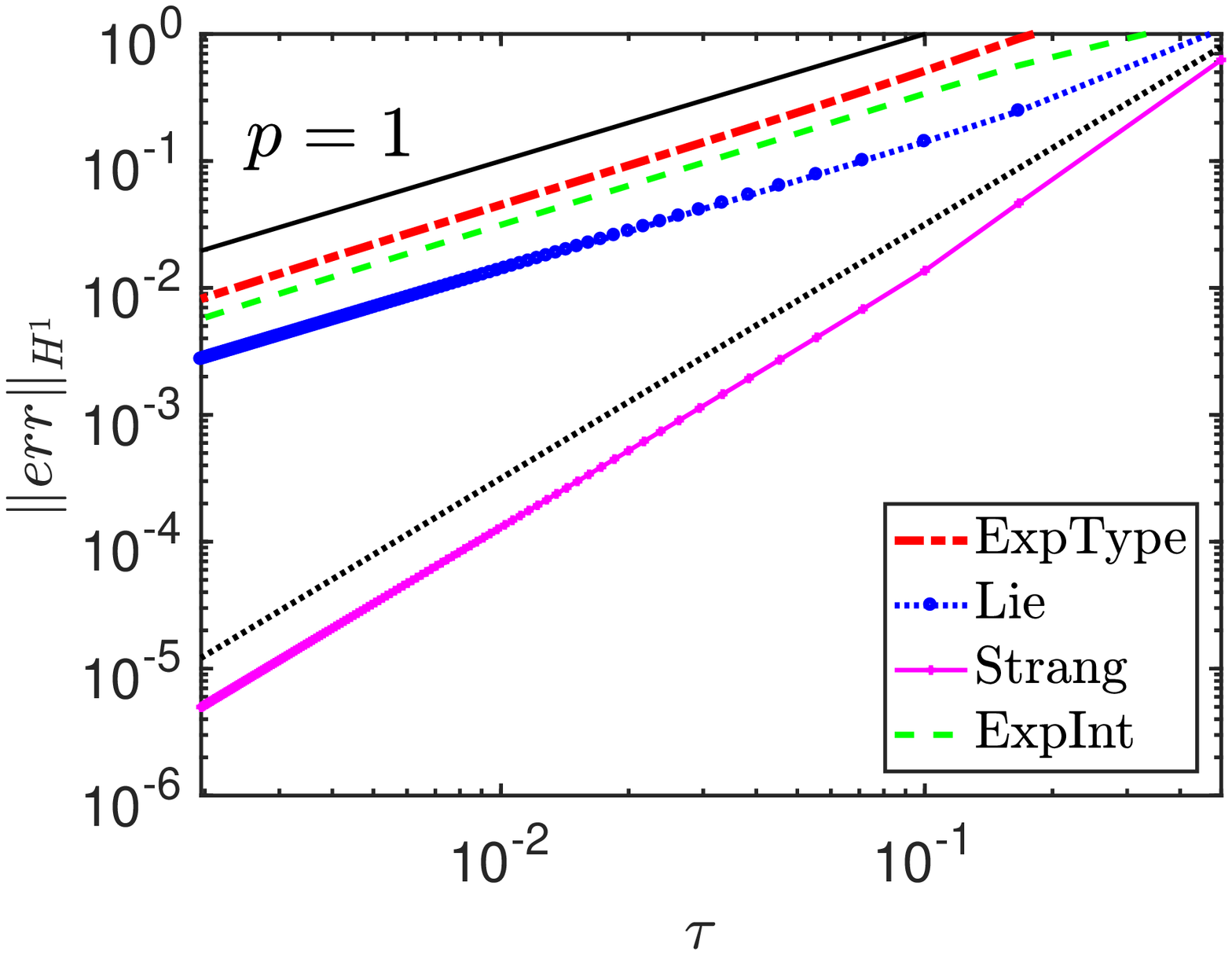}
\hfill
\includegraphics[width=0.43\linewidth]{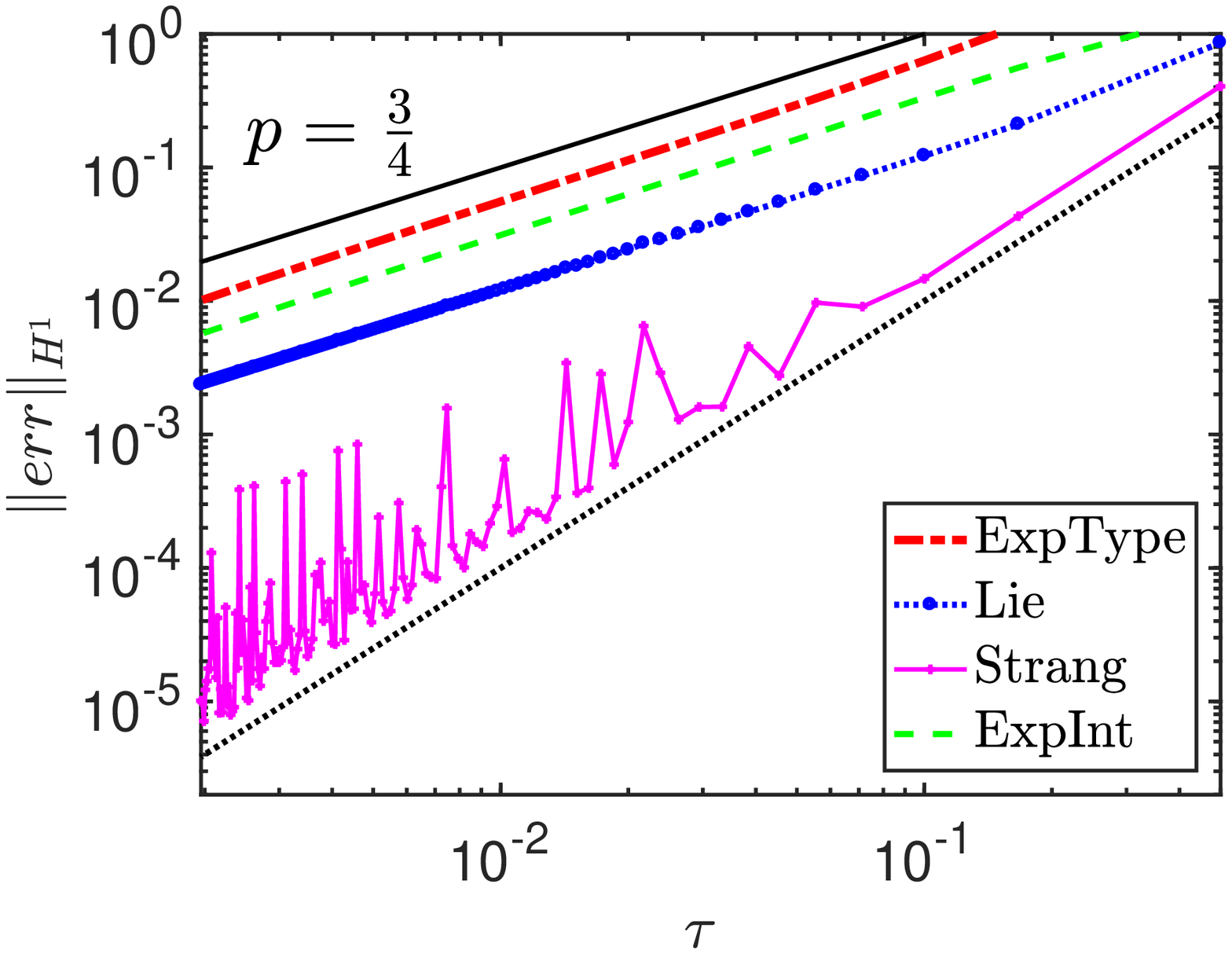}\\
\includegraphics[width=0.43\linewidth]{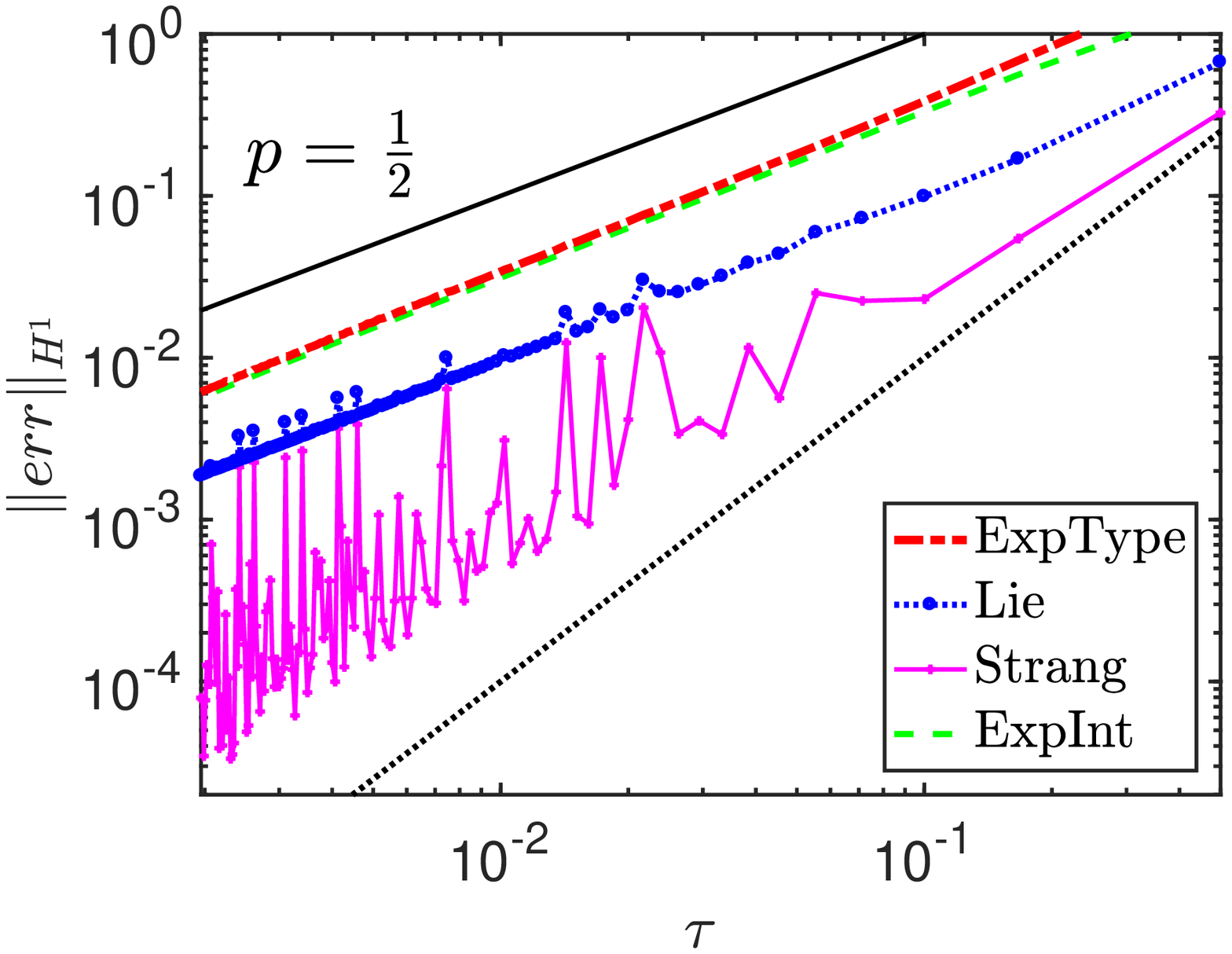}
\hfill
\includegraphics[width=0.43\linewidth]{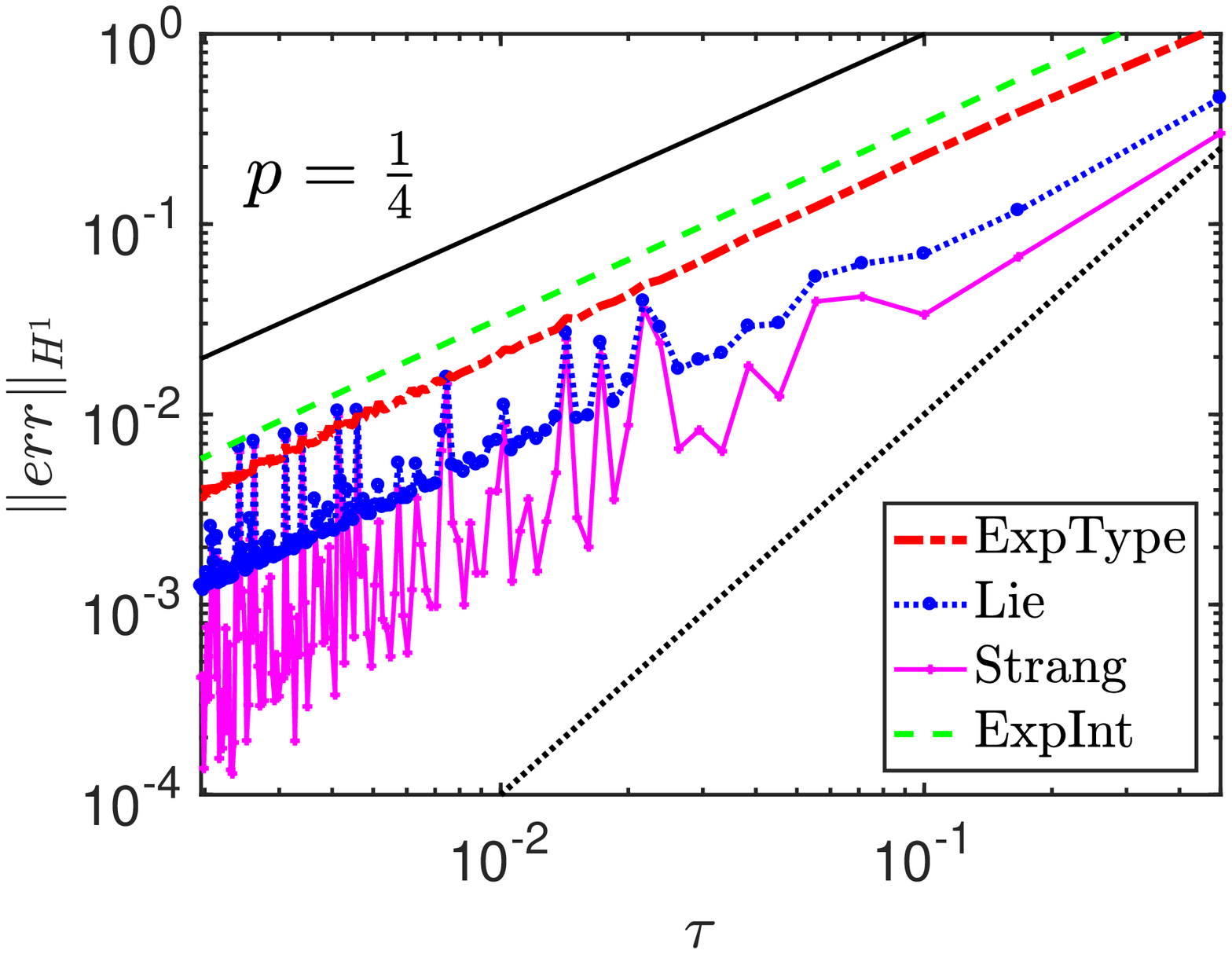}\\
\caption{Exponent $p = 1$ (upper row, left) and $p = \frac{3}{4}$ (upper row, right),  $p = \frac{1}{2}$ (lower row, left) and $p = \frac{1}{4}$ (lower row, right). Error in the energy space $H^1$ of the Lie splitting \eqref{lieP} (blue, dotted circle), Strang splitting \eqref{strangP} (magenta, star), classical exponential integrator \eqref{expiP}  (green, dashed) and exponential-type integration scheme \eqref{schemePu} (red, dash dotted) with smooth initial value \eqref{PiniD}. The slope of the  continuous and dotted line is  one and two, respectively.}\label{fig:pQuarter}
\end{figure}

\section*{Acknowledgement}
K. Schratz gratefully acknowledges financial support by the Deutsche Forschungsgemeinschaft (DFG) through CRC 1173.


\begin{thebibliography}{}

\bibitem{Bour93}
{\rm J. Bourgain},
{\em Fourier transform restriction phenomena for certain lattice subsets and applications to nonlinear evolution equations. Part I: Schr\"odinger equations.} Geom. Funct. Anal. 3:209--262 (1993).

\bibitem{BeTao06}
{\rm I. Bejenaru, T. Tao},
{\em Sharp well-posedness and ill-posedness results for a quadratic non-linear Schr\"odinger equation.} J. Funct. Anal. 233:228--259 (2006).

\bibitem{BeDe02}
{\rm C. Besse, B. Bid\'egaray, S. Descombes}, {\em  Order estimates in time of splitting methods for
the nonlinear Schr\"odinger equation}, SIAM J. Numer. Anal. 40:26--40 (2002).

\bibitem{CanG15}
{\rm B. Cano, A. Gonz\'alez-Pach\'on},
{\em Exponential time integration of solitary waves of cubic Schr\"odinger equation.} Appl. Numer. Math. 91:26--45 (2015).

\bibitem{CaWe90}
{\rm T. Cazenave, F.B. Weissler},
{\em The Cauchy problem for the critical nonlinear Schr\"odinger equation.} Non. Anal. TMA, 14:807--836 (1990).

\bibitem{CCO08}
{\rm E. Celledoni, D. Cohen, B. Owren},
{\em Symmetric exponential integrators with an application to the cubic Schr\"odinger equation.} Found. Comput. Math. 8:303--317 (2008).

\bibitem{CP11}
{\rm T. Chen, N. Pavlovi\'c},
{\em The quintic NLS as the mean field limit of a boson gas with three-body interactions.} J. Funct. Anal. 260:959--997 (2011).

\bibitem{CoGa12}
{\rm D. Cohen, L. Gauckler},
{\em One-stage exponential integrators for nonlinear Schr\"odinger equations over long times.} BIT 52:877--903 (2012).

\bibitem{Duj09}
{\rm G. Dujardin},
{\em Exponential Runge-Kutta methods for the Schr\"odinger equation.} Appl. Numer. Math. 59:1839--1857 (2009).

\bibitem{ESS16}
{\rm J. Eilinghoff, R. Schnaubelt, K. Schratz},
{\em Fractional error estimates of splitting schemes for the nonlinear Schr\"odinger equation.} J. Math. Anal. Appl. 442:740--760 (2016).

\bibitem{Faou12}
{\rm E. Faou},
{\em Geometric Numerical Integration and Schr\"odinger Equations.} European Math. Soc. Publishing House, Z\"urich 2012.

\bibitem{Gau11}
{\rm L. Gauckler},
{\em Convergence of a split-step {H}ermite method for the {G}ross--{P}itaevskii equation.} IMA J. Numer. Anal. 31:396--415 (2011).

\bibitem{GerMS09}
{\rm P. Germain, N. Masmoudi, J. Shatah},
{\em Global solutions for 3D quadratic Schr\"odinger equations.}
Int. Math. Res. Notices 3:414--432 (2009).

\bibitem{HNW93}
{\rm E. Hairer, S. P. N\o rsett, G. Wanner},
{\em Solving Ordinary Differential Equations I. Nonstiff
Problems.} Second edition. Springer, Berlin 1993.

\bibitem{HLW}
{\rm E. Hairer, C. Lubich, G. Wanner},
{\em Geometric Numerical Integration. Structure-Preserving Algorithms for Ordinary Differential Equations.} Second edition, Springer, Berlin 2006.

\bibitem{HochOst10}
{\rm M. Hochbruck, A. Ostermann},
{\em Exponential integrators.}  Acta Numer. 19:209--286 (2010).

\bibitem{HoS16}
{\rm M. Hofmanova, K. Schratz},
{\em An exponential-type integrator for the KdV equation.} Numer. Math. (2016). doi:10.1007/s00211-016-0859-1

\bibitem{HLRS10}
{\rm H. Holden, K. H. Karlsen, K.-A. Lie, N. H. Risebro},
{\em Splitting for Partial Differential Equations with Rough Solutions.} European Math. Soc. Publishing House, Z\"urich 2010.

\bibitem{Ignat11}
{\rm L. I. Ignat},
{\em A splitting method for the nonlinear Schr\"odinger equation.} J. Differential Equations 250:3022--3046 (2011).


\bibitem{KT05}
{\rm A.-K. Kassam, L. N. Trefethen},
{\em Fourth-order time-stepping for stiff PDEs.} SIAM J. Sci. Comput. 26:1214--1233 (2005).

\bibitem{KPV96}
{\rm C. Kenig, G. Ponce, L. Vega},
{\em Quadratic forms for the 1-D semilinear Schr\"odinger equation.} Trans.
Amer. Math. Soc. 346:3323--3353 (1996).

\bibitem{KPV01}
{\rm C. Kenig, G. Ponce, L. Vega},
{\em On the ill-posedness of some canonical dispersive equations.} Duke
Math. J. 106:617--633 (2001).

\bibitem{Kish09}
{\rm N. Kishimoto},
{\em Low-regularity bilinear estimates for a quadratic nonlinear Schr\"odinger equation.}
J. Differential Equations 247:1397--1439 (2009).

\bibitem{Law67}
{\rm J. D. Lawson},
{\em Generalized Runge--Kutta processes for stable systems with large Lipschitz constants.} SIAM J. Numer. Anal. 4:372--380 (1967).

\bibitem{Lubich08}
{\rm C. Lubich},
{\em On splitting methods for {S}chr\"odinger--{P}oisson and cubic nonlinear {S}chr\"odinger
equations.} Math. Comp. 77:2141--2153 (2008).

\bibitem{McLacQ02}
{\rm R.I. McLachlan, G.R.W. Quispel},
{\em Splitting methods},
Acta Numer. 11:341--434 (2002).

\bibitem{NHT01}
{\rm K. Nakanishi, H. Takaoka, Y. Tsutsumi},
{\em Counterexamples to bilinear estimates related with the KdV
equation and the nonlinear Schr\"odinger equation.} Methods Appl. Anal. 8:569--578  (2001).

\bibitem{SchO16}
{\rm K. Schratz (joint with A. Ostermann),}
{\em Derivation of a low regularity exponential-type integrator for semilinear Schr\"{o}dinger equations with polynomial nonlinearities.}
Oberwolfach Reports 18:928--931 (2016).

\bibitem{Tao01}
{\rm T. Tao},
{\em Multilinear weighted convolution of $L^2$ functions, and applications to nonlinear dispersive equations.} Amer. J. Math, 123:839--908 (2001).

\bibitem{Tao06}
{\rm T. Tao},
{\em Nonlinear Dispersive Equations. Local and Global Analysis.} Amer. Math. Soc., Providence 2006.

\bibitem{Ta12}
{\rm M. Thalhammer},
{\em Convergence analysis of high-order time-splitting pseudo-spectral methods for nonlinear Schr\"odinger equations.} SIAM J. Numer. Anal. 50:3231--3258 (2012).

\bibitem{Ts87}
{\rm Y. Tsutsumi},
{\em $L^2$ solutions for nonlinear Schr\"odinger equations and nonlinear groups.} Funk. Ekva., 30:115--125 (1987).

\end{thebibliography}
\end{document}